\documentclass[11pt]{amsart}

\usepackage{amssymb,latexsym}
\usepackage{amsthm}
\usepackage{amsmath}
\usepackage{amsfonts}
\usepackage{graphicx}
\usepackage[all]{xy}
\usepackage{chngcntr}
\usepackage{fancyhdr}
\usepackage[top=1in,bottom=1in,left=1.25in,right=1.25in]{geometry}
\usepackage[usenames,dvipsnames,pdftex]{xcolor}
\usepackage{tikz,ifthen}
\usepackage{color}
\usepackage{pgfplots}
\usepackage{tikz}
\usetikzlibrary{matrix,arrows,decorations.pathmorphing}

\newtheorem*{thm1}{Theorem}
\newtheorem{thm}{Theorem}[section]
\newtheorem{lem}[thm]{Lemma}
\newtheorem{prop}[thm]{Proposition}

\newtheorem{rem}[thm]{Remark}

\def\Q{\mathbb{Q}}
\def\F{\mathbb{F}}
\def\R{\mathbb{R}}
\def\Z{\mathbb{Z}}
\def\A{\mathbb{A}}

\usepackage[colorlinks, citecolor=blue]{hyperref}

\begin{document}

\title{$p$-adic level raising on the  eigenvariety for $U(3)$}
\date{}
\author{Ruishen Zhao}
\address{Morningside Center of Mathematics\\
	Academy of Mathematics and Systems Science\\
	Chinese Academy of Sciences\\
	No. 55, Zhongguancun East Road\\
	Beijing 100190, China}
\email{zrs13@tsinghua.org.cn}	
\renewcommand\thefootnote{}
\footnotetext{2020 Mathematics Subject Classification. Primary: 11E95; Secondary: 14G22.}

\renewcommand{\thefootnote}{\arabic{footnote}}

\maketitle

\textbf{Abstract}.
We prove  level raising results for $p$-adic automorphic forms on  definite unitary groups $U(3)/\mathbb{Q}$ and deduce some intersection points on the eigenvariety. Let $l$ be an inert prime where the level subgroups varies, if there is a non-very-Eisenstein point $\phi$ on the old component (generically parametrizing forms old at $l$) satisfying $T_{l}(\phi)=l(l^3+1)$, then this point  also lies in the new component (generically parametrizing forms new at $l$). This provides a $p$-adic analogue of Bella{\"i}che and Graftieaux's mod $p$ level raising  for classical automorphic forms on $U(3)$, and also generalizes  James Newton's    $p$-adic level raising results for definite quaternion algebras. Key ingredients include  abelian Ihara lemma (proved for any definite unitary group $U(n)$) and some duality arguments about certain Hecke modules. Finally we  also discuss some methods to construct such points explicitly and further development.

\setcounter{tocdepth}{1}
\tableofcontents


\section{Introduction}

Classical level raising results are about mod $p$ congruences  between cuspidal modular forms with different level at $l$ ($l \neq p$). A seminal example is  the following theorem proved by Ken Ribet (see  \cite{ribet1983congruence}):

\begin{thm1}

Let  $f \in S_{2}(\Gamma_0(N))$ be a normalized eigenform with level $N$, and let $\mathfrak{p}|p$ be a finite place of  $\overline{\Q}$ with $p>3$ and $f$ modulo $\mathfrak{p}$ is not congruent to an Eisenstein series. If $l \nmid Np$ is a prime with the following (level raising) condition:
\[a_l(f)^2\equiv (l+1)^2 \ (mod \  \mathfrak{p}),\]
then there exists an $l$-new eigenform  $g \in S_2(\Gamma_0(Nl))$ congruent to $f$ modulo $\mathfrak{p}$.

\end{thm1}

Here two eigenforms $f_1$ and $f_2$ are congruent modulo $\mathfrak{p}$ means that for all but finite many primes $q$, their Hecke eigenvalues are congruent, i.e. $a_{q}(f_1) \equiv a_{q}(f_2)$ after modulo $\mathfrak{p}$.

 To elaborate further on  the level raising condition, let $\pi$ denote the corresponding  automorphic representation of $GL_2(\A_{\Q})$, and let $\pi_l$ denote the $l$-part of $\pi$, which is an irreducible representation of $GL_2(\Q_l)$. Since $\pi$ is cuspidal and $\pi_l$ is unramified. Such $\pi_l$ is uniquely determined by the Satake parameter (and vice versa). A Satake parameter is called \textit{degenerate} if and only if the corresponding full principal series is reducible. For  $GL_2(\Q_l)$, this occurs exactly degenerates when  the ratio of the Satake parameter is $l$ or $l^{-1}$. Consequently, the level raising condition is equivalent to requiring that the Satake parameter is congruent modulo $p$ to a degenerate Satake parameter.

After Ribet's work, there are numerous applications and generalizations of level raising results. For instance, Andrew Wiles employed level raising results in his significant paper \cite{wiles1995} on the modularity of elliptic curves. In \cite{taylor1989galois}, Richard Taylor generalized Ribet's result to definite quaternion algebras. In \cite{bellaiche2006u3}, Jo{\"e}l Bella{\"i}che and Phillippe Graftieaux  investigated level raising congruences for definite unitary groups $U(3)$.

On the other hand, James Newton developed $p$-adic analogues of level raising results. He first considered $p$-adic automorphic forms on  definite quaternion algebras over $\Q$ in \cite{newton2011geometric}, and later  generalized those results to definite quaternion algebras over totally real number fields in \cite{newton2016level}. His approach follows the general framework of \cite{diamond1994non} and \cite{taylor1989galois} in the classical setting,  but  new phenomena arise in the context of $p$-adic setting. For instance, the spaces of relevant classical automorphic forms are finite dimensional while the $p$-adic counterparts are infinite dimensional. To address this discrepancy, James Newton introduced '\textbf{dual}' space of $p$-adic forms and established $p$-adic analogue of key duality results. Unlike classical case, Newton's Ihara lemma shows an interesting asymmetry between the usual $p$-adic space and the dual space (see lemma 7 of \cite{newton2011geometric} and section 2.10 of \cite{newton2016level}). Moreover, his level raising results implies some intersection points between old and new components of the eigenvariety. These resulting points are \textbf{non-classical} due to level raising conditions. 

In this paper, we generalize Newton's results to $p$-adic forms on  definite unitary groups $U(3)/\Q$. It  also provides a $p$-adic analogue of classical level raising results in \cite{bellaiche2006u3}. We follow the general framework of Newton. But there are some new issues in $U(3)$ setting. For example, over $\Q_l$, the group $U(3)$ has two conjugacy classes of maximal open compact subgroups, which leads to new features about unramified principal series (compare to $GL(2)$ cases). Moreover, there are endoscopy phenomena  for $U(3)$. We propose some new ideas to overcome these difficulties.

Now we quickly set-up some notations to state the strategy and describe the main theorem  (i.e. theorem \ref{main}).

Let $E$ denote an imaginary quadratic field and let $G$ denote a  unitary group $U(3)$ over $\Q$ associated with an $E$-hermitian space and $G(\R)$ is compact. Fix a prime $p$ that splits in $E$, which further induces an isomorphism between the $p$-adic reductive group $G(\Q_p)$ and $GL_{3}(\Q_p)$. Let $l$ be an odd prime  \textbf{inert} in $E$ and $G$ is unramified at $l$. For any level subgroup $\mathcal{U}$ of $G(\A_{f})$, we fix the wild level $\mathcal{U}_{p}$ as the Iwahori subgroup $Iw_{p}$ of $GL_{3}(\Q_p)$. In this paper, we focus on three level subgroups $\mathcal{U}_{0}$, $\mathcal{U}_{1}$ and $\mathcal{V}=\mathcal{U}_0 \cap \mathcal{U}_1$. They only differ at $l$-component and can be illustrated by the Bruhat-Tits building for $G(\Q_l)$, which is a  bi-homogeneous tree. More specifically:

$\bullet$ The group $\mathcal{U}_{0,l}$ is a hyperspecial subgroup of $G(\Q_l)$, corresponding to a hyperspecial vertex in the tree;

 $\bullet$ The group  $\mathcal{U}_{1,l}$ is another kind of maximal open compact subgroup of $G(\Q_l)$, corresponding to an adjacent special vertex;

$\bullet$ The group $\mathcal{V}_{l}=\mathcal{U}_{0,l} \cap \mathcal{U}_{1,l} $ is the resulting Iwahori subgroup.

Let $\mathcal{W}$ denote the weight space. For an admissible open irreducible affinoid $X \hookrightarrow \mathcal{W}$, we investigate three Hecke modules \[L_0=\mathcal{S}_X(\mathcal{U}_0,r)^{Q}, \  L_1=\mathcal{S}_X(\mathcal{U}_1,r)^{Q} \  and \ M=\mathcal{S}_X(\mathcal{V},r)^{Q}.\] Here $r=p^{-m}$ ($m \in \mathbb{N}_{+}$) is the convergent radius  and $Q$ is a (slope-truncation) polynomial related with Fredholm theory of Hecke operators at $p$. And $\mathcal{S}_X(\mathcal{V},r)$ is the space of $p$-adic forms on $G(\A_f)$ with weight $X$, level $\mathcal{V}$ and locally $r$-analytic. The Hecke module $M$ is a direct summand of it cutting out by slope decomposition. Other Hecke module $L_0$ and $L_1$ have similar meaning. There are two kinds of natural maps:

$\bullet$ A \textit{level raising} map  $i:L_0\oplus L_1 \longrightarrow M$, whose image $im(i)$ is the space of $p$-adic forms \textit{old} at $l$.

$\bullet$ A  \textit{level lowering} map $i^+:M\longrightarrow L_0\oplus L_1$,  whose  kernel $ker(i^+)$  is the space of  $p$-adic forms \textit{new} at $l$.

The level raising problem can be interpreted as  comparing the support (over related tame Hecke algebras) between modules of old forms and new forms.


We go through a similar routine by James Newton and make a comparison with his ideas. Key ingredients are similar, Ihara lemma and some duality arguments, while new features appear for definite unitary groups. James Newton works with definite quaternion algebras,  the group is  $GL(2)$ over the place $l$. Its Bruhat-Tits tree is homogeneous, thus simpler than our setting. As a result, he mainly dealt with two kinds of level subgroups, $\mathcal{U}$ (hyperspecial at $l$) and $\mathcal{V}$ (Iwahori at $l$). And for old forms, he only needed to consider modules like $L_0^2$. But we need to consider three level subgroups and consider $L_0\oplus L_1$ instead of $L_0^2$. This is a little different from the definition of old forms in classical setting of \cite{bellaiche2006u3}. As a result, our level changing matrix is also different from classical case (see section \ref{oldandnew}). In the last part of this paper (section \ref{localpicture}), we explain some (local) intuition why it is better to use $L_0 \oplus L_1$ instead of $L_0^2$. To study level raising problems, we also introduce dual modules and consider a natural pairing. To apply duality argument, we need to first verify that $ker(i^+i)=0$. Newton deduced such injectivity by Ramanujan-Petersson conjecture for cuspidal Hilbert modular forms (see  proposition 2.13 in \cite{newton2016level}). In our setting, because there are some components of eigenvarieties for $U(3)$ coming from endoscopy, we can't apply Ramanujan conjecture to such components (generically parameterize endoscopic forms, thus not cuspidal). Instead, we first deduce a kind of Ihara lemma. We call it \textbf{abelian Ihara lemma} (theorem \ref{ihara}). In fact, this lemma holds for a general class of reductive groups including any definite unitary group $U(n)$ ($n\geq2$), see section \ref{ihara} for more details. This Ihara lemma shows  strict restrictions  for abelian $p$-adic forms. Combine it with Zariski density of classical points and certain ('semisimple') property of classical forms, we get the desired injectivity of $i^+i$. Then we apply similar duality arguments to study $ker(i^+)$. Apply abelian Ihara lemma again, we deduce $p$-adic level results. And it implies some intersection points between old components and new components inside the eigenvariety.

The following theorem (see theorem \ref{main}) is our main theorem:

\begin{thm1}
Let $\mathcal{E}(\mathcal{U}_0)$ denote the eigenvariety for $G$ with level $\mathcal{U}_0$. Suppose we have a point $\phi$ on it which is not very Eisenstein and $T_l(\phi)=l(l^3+1)$. Then the corresponding point inside the old component $\mathcal{E}(\mathcal{V})^{old}$ also lies in the new component $\mathcal{E}(\mathcal{V})^{new}$.

\end{thm1}

After that we discuss how to construct such points. In \cite{newton2016level}, Newton constructed such points by some explicit computations about Hida families. Because $GL(2)$ is also closely related with $U(2)$, we can try to first transfer Newton's points to  $U(2)$ eigenvariety. Then use $p$-adic Langlands functoriality (symmetric square) to get such points on $U(3)$ eigenvariety. These points are not classical, thus we can't apply symmetric power functoriality for classical forms (e.g. \cite{nt2021symmetric1} and \cite{nt2021symmetric2}) directly. We need to do $p$-adic interpolation for such functoriality (see \cite{hansen2017universal}). Finally, we discuss further development. The abelian Ihara lemma should work for any   reductive group $G$ over $\Q$ with $G(\R)$ being compact and $G^{der}$ being simply connected. The main theorem should also work for such group  if further $G(\Q_l)$ has (reduced) rank one. For higher rank, such level raising problems become much more difficult. In fact,  in classical setting, Clozel, Harris and Taylor proposed a conjecture about generalizations of Ihara lemma to definite unitary groups $U(n)$ (over \textbf{split} primes) in \cite{clozel2008automorphy}. This is still open when $n>2$. The $p$-adic analogue seems harder. Moreover, in higher rank, there are more kinds of components (besides 'new'  and 'old') inside the eigenvariety. Here we instead discuss local analogues for $GL(n)$ about intersection points on the moduli space of tame $L$-parameters. We also point our some new features for other groups (like $SL(2)$ and $U(3)$). Although the local analogue is not needed to prove results in global setting, the reader may read it (section \ref{localpicture}) first to get better geometric motivations.

We briefly describe the structure of this paper. In section \ref{notion}, we introduce basic notions of $p$-adic overconvergent forms and its dual module on definite unitary groups $U(n)$. In section \ref{sectionihara}, we prove  the abelian Ihara lemma. Then we turn to $n=3$ case. In section \ref{sectionlevelraising}, we introduce the natural pairing. After introducing old and new forms, we get some duality results. Finally we deduce a kind of $p$-adic  level raising result. In the last section \ref{sectioneigenvariety}, we use these results to deduce the main theorem concerning intersection of irreducible components on $U(3)$ eigenvariety. After that, we propose a method to construct such intersection points and discuss further development. Then we develop some local analogues for $GL(n)$ and discuss some new features for other groups.

\textbf{Acknowledgments} First I want to thank Yiqin He for his help. I'm also grateful to James Newton and Jack Thorne for some useful conversations. Besides, I thank Yichao Tian and Zhixiang Wu. Finally, this paper is  dedicated to Jo{\"e}l Bella{\"i}che. I learned eigenvarieties from his wonderful lecture notes. It is really a pity that I don't have a chance to discuss these things with him.

\section{$p$-adic forms on the definite unitary groups}
\label{notion}

\subsection{basic notations}

Let $E$ be an imaginary quadratic field extension of $\Q$ and $D$ denote the central simple $E$-algebra $M_{n}(E)$ (here $n\geq 2$). Take an involution of $D$, $x\mapsto x^*$, extending the nontrivial automorphism $\sigma$ of $E$ over $\Q$ (for example, we can take it to the adjunction with respect to a non-degenerate Hermitian form). Let $G/\Q$ denote the unitary group whose $R$-points (for any $\Q$-algebra $R$) are \[G(R)=\{x \in D\otimes_{\Q}R|x x^*=1\}.\] Then if a prime $p$ is split in $E$, we have $G(\Q_p)\cong GL_{n}(\Q_p)$ and $G(\R)\cong U_{s,t}(\R)$. From now on, we fix a split prime $p$ and further assume the signature $(s,t)$ is $(0,n)$ or $(n,0)$. In particular, $G(\R)$ is compact and we call such a group $G$ \textit{definite unitary group}.

  Let $\A$  denote $\A_{\Q}$, $\A_{f}$ denote the finite adeles of $\A$, and $\A_{f}^p$ denote the finite adeles trivial at $p$. Let $\mathcal{U}$ be a compact open subgroup of $G(\A_f)$ of the form $\mathcal{U}_{p}\times \mathcal{U}^{p}$, where $\mathcal{U}_p$ is a compact open subgroup of $G(\Q_p)$ (\textit{wild level}) and $\mathcal{U}^p$ is a compact open subgroup of $G(\A^p)$ (\textit{tame level}). Then for any commutative ring $R$ and any \textbf{right} $\mathcal{U}_p$-module $A$ over $R$, we can define an $R$-module ($A$-valued automorphic forms)  $\mathcal{F}(\mathcal{U},A)$ in the following way:

\[\mathcal{F}(\mathcal{U},A)=\{f:G(\Q)\backslash G(\A_f)\longrightarrow A,\ f(gu)=f(g)u_p \ for\  all\ u\in \mathcal{U}\}.\]

Here we follow the convention of James Newton (see \cite{newton2011geometric} and \cite{newton2016level}) and mainly use \textbf{right} action.  This convention is slight different from others, like \cite{bc2009family}, where they mainly used left action.

For any function $f:G(\A_f)\longrightarrow A$ and $x \in G(\A_f)$ with $x_p \in \mathcal{U}_p$, we define  a new function  $f|x:G(\A_f)\longrightarrow A$ by \[(f|x)(g)=f(gx^{-1})x_p.\] Then we can also write the above module as \[\mathcal{F}(\mathcal{U},A)=\{f:G(\Q)\backslash G(\A_f)\longrightarrow A,\ f|{u}=f \ for \ all \ u \in \mathcal{U}\}.\] By generalized finiteness of class groups (e.g. see \cite{borel1963some}), the double coset $G(\Q)\backslash G(\A_f)/\mathcal{U}$ is finite. Pick up a set of representatives $\{x_i|1\leq i\leq h\}$ for this double coset, we have the following isomorphism \[\mathcal{F}(\mathcal{U},A)\longrightarrow \bigoplus_{i=1}^{h}A^{x_i^{-1}G(\Q)x_i \cap \mathcal{U}},\] \[f\mapsto (f(x_1),...,f(x_h)).\] Moreover, each $x_i^{-1}G(\Q)x_i \cap \mathcal{U}$ is a finite group, and it is trivial if the tame level $\mathcal{U}^p$ is small enough. For example, see proposition 4.1.1 of \cite{chenevier2004familles}. From now on, we assume that the tame level $\mathcal{U}^p$ is small enough (\textit{neat}). Then $\mathcal{F}(\mathcal{U},A)\cong A^{h}$. Although in fact such neat assumption can be removed, it will simplify some computations (such as verifying adjoint property of Hecke operators under the pairing in section \ref{sectionpairing}).

We can define some double coset operators (\textit{Hecke operators}) on this module. Here for simplicity we illustrate \textit{tame} Hecke operators. Let $\mathcal{U}_0$ and $\mathcal{U}_1$ denote two level subgroups of $G(\A_f)$ with the same wild level subgroup. For any $x \in G(\A_f)$ with $x_p=1$, we can define an $R$-linear map \[[\mathcal{U}_0x\mathcal{U}_1]:\mathcal{F}(G,\mathcal{U}_0)\longrightarrow \mathcal{F}(G,\mathcal{U}_1)\] as follow: first decompose $\mathcal{U}_0x\mathcal{U}_1$ into a finite disjoint union $\coprod_{i} \mathcal{U}_0g_i$ and define \[f|[\mathcal{U}_0x\mathcal{U}_1]=\sum_i f|g_i.\] In particular, if $\mathcal{U}_0=\mathcal{U}_1$, this module endows a right action by tame Hecke algebras. It is more subtle to define Hecke actions at $p$.  We should be care about the range of $x_p$, because usually the module $A$ doesn't have an action by the whole group $G(\Q_p)$. Instead we will restrict to elements inside a monoid of $G(\Q_p)$. Then the double coset operators acts on this module in the same way. As this paper is mainly about tame information of $p$-adic forms, we refer to chapter 7 of \cite{bc2009family} for more details about Hecke operators at $p$.

Now we turn to discuss wild level. Recall that $p$ splits in $E$, therefore $G(\Q_p)\cong GL_n(\Q_p)$. We fix such an isomorphism. Write $B$ and $\overline{B}$ for the upper and lower triangular Borel subgroups respectively, $N$ and $\overline{N}$ for the upper and lower unipotent subgroup of $GL_n$ respectively, and $T$ for the diagonal torus. Moreover, let $Iw_p$ denote the Iwahori subgroup of $GL_n(\Q_p)$ defined by such $B$ and $GL_n(\Z_p)$, i.e. $Iw_p$ is the subgroup of $GL_n(\Z_p)$ that becomes the upper triangular Borel subgroup after modulo $p$.  From now on, we fix the wild level $\mathcal{U}_p$ as $Iw_p$.

To define $p$-adic   automorphic forms, we need to construct suitable $Iw_p$-module first. It is usually constructed by certain induction methods.  We will use the following notations for induction:

If $B_1\subset H_1$ are groups, $R$ is a commutative ring, and $\chi:B_1\longrightarrow R^*$ is a character, let
\[Ind_{B_1}^{H_1}\chi=\{f:H\longrightarrow R| f(hb)=f(h)\chi(b)\ for \ all \ h\in H_1,\ b\in B_1\}.\] What's more, if $Pro$ is a property for some functions $f \in Ind_{B_1}^{H_1}\chi$ that is invariant under left translation by $H_1$, then let \[Ind_{B_1}^{H_1,Pro}\chi=\{f \in Ind_{B_1}^{H_1}\chi | f \ has \ property \ Pro\}.\] Then $Ind_{B_1}^{H_1,Pro}$ is an $R$-module with a right action of $H_1$ given by left translations, i.e. $(f.h)(x)=f(hx)$ for all $h,x\in H_1$.

For example, let $R$ denote a $p$-adic field, $\chi=(t_1,...,t_n)\in \Z^n$ and we write $diag(d_1,..,d_n)$ for the diagonal matrix, we can interpret $\chi$ as the character of the diagonal  torus $T(R)$ mapping $diag(d_1,...,d_n)$ to $\prod_{i}d_i^{t_i}$ and thus also view it as the character of the upper triangular Borel subgroup $B(R)$ by reducing to $T(R)$ and applying $\chi$. Assume $t_1\geq t_2...\geq t_n$, the $R$-vector space \[Ind_{B(R)}^{GL_n(R),alg}\chi,\] where $alg$ means \textit{algebraic}, is the irreducible algebraic representation of $GL_n$ over $R$ with highest weight $\chi$. Let $w_0$ denote the longest element of the Weyl group, then \[\mathcal{F}(\mathcal{U},Ind_{B(R)}^{GL_n(R),alg}w_0(-\chi))\] is the space of classical (algebraic) automorphic forms on $G$ of weight $\chi$ and level $U$ with coefficients in $R$. See section 7.3.5 of \cite{bc2009family} for more details.

\subsection{weight space}

A weight $\chi$ is a $p$-adic continuous character of $T(\Z_p)\cong(\Z_p^*)^n$. Or Equivalently, we can write $\chi$  as $(\chi_1,...,\chi_n)$, each $\chi_i$ is a $p$-adic continuous character of $\Z_p^{*}$, and $\chi$ is defined by sending $diag(a_1,...,a_n)$ to $\prod_{i}\chi_{i}(a_i)$. Similar to the algebraic example, we can also view such a weight as a character of $B(\Z_p)$ by reducing to $T(\Z_p)$. The weight space $\mathcal{W}$ is the rigid analytic space over $\Q_p$  such that for any $\Q_p$-affinoid algebra $R$, $\mathcal{W}(R)$ is the set of continuous characters $(\Z_p^*)^n\longrightarrow R^*$. Let $\triangle=((\Z/p)^*)^{n}$ and we have \[(\Z_p^*)^{n}\cong \triangle \times (1+pZ_p)^n.\] Thus any $R$-point of $\mathcal{W}$ is determined by a character of $\triangle$ and a character of $(1+pZ_p)^n$. In geometry, the weight space  is a finite disjoint union of $n$-dimensional open unit polydiscs (or called balls).

For any $\Q_p$-affinoid algebra $R$, and a continuous character $\chi_1:\Z_p^*\longrightarrow R^*$, it is called locally $r$-analytic (here $r=p^{-m}$ with $m\in \mathbb{N}_{+}$), if its restriction to $1+p^m \Z_p$ can be given by a convergent power series with coefficient in $R$. For simplicity in this paper we will always work with radius $r$ in the form of $p^{-m}$. By section 7.3.3 of \cite{bc2009family}, we can always find such convergent radius $r$ for $\chi_1$. For a weight $\chi=(\chi_1,...,\chi_n)$, we call it locally $r$-analytic if each $\chi_i$ is locally $r$-analytic. Or equivalent, its restriction on $(1+p^m\Z_p)^n$  is given by a convergent power series. Such radius $r$ always exists.

For any  reduced $\Q_p$-affinoid  $X$ with a morphism $X\longrightarrow \mathcal{W}$, we use $[\cdot]_{X}$ to denote the resulting weight \[[\cdot]_X:(\Z_p^{*})^{n}\longrightarrow O(X)^{*}. \]

\subsection{overconvergent forms}
\label{sectionpadicforms}
Now we will construct suitable $Iw_p$ module to define overconvergent forms. The first step is to introduce the general space $Ind_{B(\Z_p)}^{Iw_p}[\cdot]_X$, the second step is to define locally $r$-analytic submodule $Ind_{B(\Z_p)}^{Iw_p,r-an}[\cdot]_X$ and the final step is to apply the functor $\mathcal{F}(\mathcal{U},-)$ to get the global module of locally $r$-analytic forms.

Recall that we can also view the weight $[\cdot]_X$ as a character of $B(\Z_p)$, so we can consider the following induction $Ind_{B(\Z_p)}^{Iw_p}[\cdot]_X$. Notice that we have the following isomorphism (via natural inclusions)\[\overline{N}(\Z_p)\cong Iw_{p}/B(\Z_p)\cong Iw_{p}B(\Q_p)/B(\Q_p),\] so this $O(X)$-module has extra $B(\Q_p)$ right action and we can identify it (through restriction on $\overline{N}(\Z_p)$) with the space of $O(X)$-valued functions on $\overline{N}(\Z_p)$.

Moreover, we have the following identification \[\Z_{p}^{\frac{n(n-1)}{2}}\cong \overline{N}(\Z_p),\] \[\underline{z}=(z_{i,j})\mapsto \overline{N}(\underline{z})=
\begin{pmatrix}
1 & 0 & \cdots & 0 \\
pz_{2,1} & 1 & \cdots & 0 \\
\vdots & \vdots & \vdots & \vdots \\
pz_{n,1} & \cdots & pz_{n,n-1} & 1 \\
\end{pmatrix}.\]

We put the dictionary order on the index set $Inx=\{(i,j)|1\leq j < i \leq n\}$, so we can also present $\underline{z}$ by a tuple $(z_{2,1},...,z_{n,n-1})$. For later application in Ihara lemma (section \ref{sectionihara}), we further introduce the root subgroup $R_{2,1}$: \[\Z_p \hookrightarrow \overline{N}(\Z_p),\] \[a\mapsto \overline{N}((a,0,\cdots,0)).\] It is easy to see that this is a group map, $R_{2,1}(a+b)=R_{2,1}(a) R_{2,1}(b)$. What's more, for any $\underline{z}$, set $\widetilde{z}=\underline{z}-(z_{2,1},0,\cdots,0)$, then $\overline{N}(\underline{z})=R_{2,1}(z_{2,1})N(\widetilde{z})$.

For a function $f:\overline{N}(\Z_p)\longrightarrow O(X)$, we call it locally $r$-analytic (recall $r=p^{-m}$), if for any $\underline{a}=(a_{i,j})\in \Z_{p}^{\frac{n(n-1)}{2}}$, the restriction of $f$ on \[Ball(\underline{a},r)=\{\underline{z}=(z_{i,j})\in \Z_{p}^{\frac{n(n-1)}{2}}| z_{i,j} \in a_{i,j}+p^m \Z_{p}\}\] is given by a convergent power series with coefficients in $O(X)$.

Let $r-an$ denote the property being locally $r$-analytic,  we introduce the $O(X)$ module \[A_{X,r}=Ind_{B(\Z_p)}^{Iw_p,r-an} [\cdot]_X.\] It still has the right $Iw_p$ action. But be careful, the previous $B(\Q_p)$ action may not keep the property $r-an$. Instead there exists a monoid $\mathcal{M}$ of $B(\Q_p)$ that keeps the convergent radius $r$. Therefor $A_{X,r}$ has extra $\mathcal{M}$ action. Through this monoid action, we can define suitable Hecke operators at $p$ and slope etc, which is very important for $p$-adic forms. We refer to section 7.3 of \cite{bc2009family} for more details.

To prove the abelian Ihara lemma, we explore more about $A_{X,r}$. From the definition, we have the following isomorphism \[A_{X,r}\cong \oplus_{k} O(X)\langle Z_{i,j} \rangle, \] this is a finite sum with $p^{m\frac{n(n-1)}{2}}$ components. This is done by picking up a set of representatives for $\Z_{p}^{\frac{n(n-1)}{2}}/ (p^m \Z_p)^{\frac{n(n-1)}{2}}$, and identify the closed ball of radius $r$ with the unit closed ball through rescaling. Thus as an $O(X)$-module, $A_{X,r}$ is just some copies of standard Tate algebras. For each summand, we write an element $Z \in O(X) \langle Z_{i,j} \rangle$ as \[Z=\sum_{\alpha}x_{\alpha}Z^{\alpha},\] here $\alpha=(\alpha_{i,j})$ runs over $\mathbb{N} ^{Inx}$, $Z^{\alpha}$ is short for $\prod_{i,j} Z_{i,j}^{\alpha_{i,j}}$, $x_{\alpha} \in O(X)$ and tends to $0$.

We define the space of \textit{locally} $r$-\textit{analytic} (or $r$-\textit{overconvergent}) $p$-adic automorphic forms on $G$ with level $U$ and weight $X$ as the following $O(X)$-module: \[\mathcal{S}_X(\mathcal{U},r):=\mathcal{F}(\mathcal{U},A_{X,r}).\]

\subsection{dual modules}
To study new forms, we will use some duality arguments. Follow \cite{newton2011geometric} and \cite{newton2016level}, we need to define dual module of $p$-adic forms.

For any Banach algebra $R$ and Banach $R$-module $A$, we define the \textit{dual module} $A^*$ as the Banach $R$-module of \textbf{continuous} $R$-linear maps from $A$ to $R$, with the usual operator norm. For instance, let $R$ be $\Q_p$ and $A$ be $R\langle T\rangle$, the standard Tate algebra. Then its dual $A^*$ can be identified with the module of power series with bounded coefficient $R[\langle T \rangle]=\Z_p[[T]]\otimes \Q_p$. In particular, $A$ is an ONable module (has an orthonormal basis) while $A^*$ is not (see \cite{newton2011geometric} section 2.1).

For simplicity, we denote the $O(X)$-module $A_{X,r}^*$ as $D_{X,r}$.  Notice that $A_{X,r}$ has an extra right actions by certain monoid $\mathcal{M}$ inside $B(\Q_p)$, we endow $D_{X,r}$ with right actions by $\mathcal{M}^{-1}$ through the dual action. In other words, for any $g\in \mathcal{M}$, $x \in A_{X,r}$ and $\lambda \in D_{X,r}$, we have \[(x.g,\lambda)=(x,\lambda.g^{-1}),\] here $(-,-)$ is the natural pairing.

What's more, in the previous section, we have explicit $A_{X,r}$ as a finite sum of standard Tate algebras, then correspondingly we have the following isomorphism: \[D_{X,r}\cong \oplus_{k} O\langle X\rangle [\langle C_{i,j} \rangle].\] Similarly, for each summand and an element $C$ of it, we write \[C=\sum_{\alpha}y_{\alpha}C^{\alpha},\] and the under the natural pairing we have \[(Z,C)=\sum_{\alpha}x_{\alpha} y_{\alpha}.\]

We define the dual space of locally $r$-analytic $p$-adic forms  on $G$ with level $\mathcal{U}$ and weight $X$ as the following $O(X)$-module
\[\mathcal{D}_{X}(\mathcal{U},r):=\mathcal{F}(\mathcal{U},D_{X,r}).\] Later through the natural pairing (see section \ref{sectionpairing}), we will identify $\mathcal{D}_{X}(\mathcal{U},r)$ with the $O(X)$-dual module of $\mathcal{S}_{X}(\mathcal{U},r)$.

\subsection{slope decomposition and  Hecke actions}
\label{sectionslope}
Similar to Newton's \cite{newton2016level} section 2.7, in practice we will use finite slope truncation to get some finitely generated submodules. This finiteness property is very useful to do duality arguments (see section \ref{duality}).

 Pick a suitable Hecke operator $U_{h}$ at $p$ (for example see section 7.3.6 of \cite{bc2009family}), there is a power series $P_{h}(T)\in O(X)[[T]]$ corresponding to the determinant polynomial for $U_{h}$. This power series is entire (convergent on the whole $\mathbb{A}^1$) and it is a \textit{Fredholm series}. Through  Fredholm theory, we have a factorization $P_{h}(T)=Q(T)S(T)$ where $Q(T)$ is a polynomial of degree $m$ with unit leading coefficient and it is prime to $S(T)$. Denote by $\widetilde{Q}(T)=T^mQ(T^{-1})$, then there is a canonical slope decomposition \[\mathcal{S}_X(\mathcal{U},r)=\mathcal{S}_X(\mathcal{U},r)^{Q}\oplus Oth,\] where $\widetilde{Q}(U_{h})$ is zero on  $\mathcal{S}_X(\mathcal{U},r)^{Q}$ and invertible on $Oth$. Moreover, $\mathcal{S}_X(\mathcal{U},r)^{Q}$ is a finite projective $O(X)$-module and indeed independent of the (small enough) radius $r$. Fix the factorization $P_{h}(T)=Q(T)S(T)$, then this decomposition commutes with base change, for any reduced affinoid $Y\longrightarrow X$, we have \[\mathcal{S}_{Y}(\mathcal{U},r)^{Q}\cong \mathcal{S}_X(\mathcal{U},r)^Q \widehat{\otimes}_{O(X)} O(Y).\] Again we refer to section 7.3 of \cite{bc2009family} for more details.

Notice that the tame Hecke action commutes with Hecke action at $p$, obviously the $O(X)$ module $S_X(U,r)^Q$ also has tame Hecke action. Before going on, we set up some notations for Hecke algebras.

Let $S_0$ denote a finite set of primes including $p$ such that for any prime $q \notin S_0$, $q$ is unramified in $E$ and $G(\Q_q)$ is unramified over $\Q_q$ and $U_{q}$ is a maximal open compact subgroup of $G(\Q_q)$. What's more, we always assume $\mathcal{U}=\mathcal{U}_{S_0}\times \prod_{q \notin S_0}\mathcal{U}_q$. In this paper we will mainly consider local Hecke algebra at such $q \notin S_0$.

For any prime $q\notin S_0$ that splits in $E$, suppose $q=\beta_0 \beta_1$. Then the inclusion $\beta_0: E\hookrightarrow \Q_{q}$ (with $E_{\beta_0}\cong \Q_q$) induces an isomorphism $G(\Q_q)\cong GL(n,\Q_q)$ and  another isomorphism $G^{ab}(\Q_q)\cong \Q_q^*$ for the abelian quotient $G^{ab}=U(1)$. Then we can consider the additive valuation $v_q$ on $G^{ab}(\Q_q)$. Let $\pi_q$ denote a uniformizer of $\Q_q$ (e.g. $\pi_q=q$), through this isomorphism, we also let $\pi_q$ denote the resulting element in $G^{ab}(\Q_q)$. Further combine with the natural inclusion $G^{ab}(\Q_q)\hookrightarrow G^{ab}(\A_f)$, we will also view it as an element of the later group.

For any element $x \in GL(n,\Q_q)$, let $T_{q,x}$ denote characteristic function on the double coset $\mathcal{U}_qx\mathcal{U}_q$ (notice that $\mathcal{U}_q$ is conjugated to $GL(n,\Z_q)$), thus become an element inside the spherical Hecke algebra (valued in $\Z$) for $GL(n,\Q_q)$. Denote this algebra by $\mathcal{H}_q$. If we enlarge the coefficient ring into $\Q(q^{\pm\frac{1}{2}})$, then we can apply the usual Satake isomorphism to explicit this ring structure. And for the group $GL(n,\Q_q)$ with standard hyperspecial subgroup $GL(n,\Z_q)$, we usually use these Hecke operators $\{T_{q,x_i}|x_i=\mu_i(\pi_q)\}$ (here $\mu_i$ runs over minuscule cocharacters), see \cite{Gross1998} for more details.  Now  we let $T_{q,x}$ act on $\mathcal{S}_X(\mathcal{U},r)^Q$ (or $\mathcal{S}_X(\mathcal{U},r)$) through the double coset action by $[\mathcal{U}x\mathcal{U}]$. Through this way $\mathcal{H}_q$ acts on the module of $p$-adic forms. For the dual side, we let $T_{q,x}$ act on  $\mathcal{D}_X(U,r)$ through $[\mathcal{U}x^{-1}\mathcal{U}]$. Later through a natural pairing (see section \ref{sectionpairing}), we will identify these $\mathcal{D}_{X}(\mathcal{U},r)$ exactly as the dual space of  $\mathcal{S}_{X}(\mathcal{U},r)$. Then the later action of $\mathcal{H}_q$ is exactly the dual action. In other words, we have $(t.f,\lambda)=(f,t.\lambda)$ for any $t \in \mathcal{H}_q$, $f \in \mathcal{S}_X(\mathcal{U},r)$ and $\lambda \in \mathcal{D}_X(U,r)$. For a prime $q$ that is inert in $E$ and $q \notin S_0$, we can define the Hecke algebra and tis action similarly.

For the prime $p$, we use $\mathcal{H}_p^{-}$ denote the algebra of related Hecke operators at $p$ (not the whole Iwahori-Hecke algebra). This construction is subtle and use the previous mentioned monoid $\mathcal{M}$ inside $B(\Q_p)$. We refer to chapter 7 of \cite{bc2009family} for more details. Let $X$ (and $Q$) vary, through the study of Hecke actions on $\mathcal{S}_X(\mathcal{U},r)^{Q}$ and apply eigenvariety machine (roughly speaking via gluing image of Hecke algebras), we can construct an eigenvariety $\mathcal{E}(\mathcal{U})$ with level $\mathcal{U}$ and a natural weight map $\mathcal{E}(\mathcal{U})\longrightarrow \mathcal{W}$.  See chapter 7 of \cite{bc2009family}. Also see \cite{ludwig2024spectral} for the general machine of constructing eigenvarieties.

\section{Abelian Ihara lemma}
\label{sectionihara}

Let $\det$ denote the natural map to abelian quotient $G\longrightarrow G^{ab}$ and we will also use it to denote the map $G(\A_f)\longrightarrow G^{ab}(\A_f)$ etc. Let $Y$ be an irreducible reduced affinoid with a map $Y\longrightarrow \mathcal{W}$. In this section, we will prove the abelian Ihara lemma. It concerns about abelian forms, i.e. elements $f$ of $\mathcal{S}_Y(\mathcal{U},r)$ or $\mathcal{D}_{Y}(\mathcal{U},r)$ that factors through (think it as a function on $G(\A_f)$) the abelian quotient $G^{ab}(\A_f)$ through the map $\det$.

For any prime  $q \notin S_0$, and an element $x_q \in G(\Q_q)$, consider the local Hecke operator $T_{q,x_q}$ for the double coset $\mathcal{U}_q x_q \mathcal{U}_q$, we denote the number $deg(T_{q,x_q})$ for the cardinality of $(x_q^{-1}\mathcal{U}_qx_q \cap \mathcal{U}_q) \backslash \mathcal{U}_q$.

Here we  compute some examples of this $deg$ function. If $q \notin S_0$ and $q$ splits in $E$, we know that $\mathcal{U}_q$ is conjugated to $GL(n,\Z_q)$. To simplify notations, we assume it is exactly $GL(n,\Z_q)$. Take $x_q=diag(q,...,q,1,...,1)$ ($q$ appears $i$ times with $1\leq i \leq n$), then $deg(T_{q,x_q})=\sharp Gr(n,n-i)(\F_{q})=\sharp Gr(n,i) (\F_q)$. Here $\sharp$ means the cardinality, $\F_{q}$ is the finite field with $q$-elements and $Gr(n,i)$ means the Grassmannian parameterizes $i$-dimensional subspace inside a fixed $n$-dimensional space. If $i=1$ or $i=n-1$, it is just the projective space, and the cardinality is $\frac{q^{n}-1}{q-1}$.

\begin{thm}
\label{ihara}
Let $Y$ be an irreducible reduced affinoid with a map $Y\longrightarrow \mathcal{W}$.

(1) If $\lambda \in \mathcal{D}_{Y}(\mathcal{U},r)$ factors through the map $\det$, then $\lambda=0$.

(2) If $f \in \mathcal{S}_{Y}(\mathcal{U},r)$ factors through the map $\det$ and $f$ is nonzero,  write weight $[\cdot]_Y$ as $\chi=(\chi_1,...,\chi_n)$,  then $\chi$ is \textbf{central}: $\chi_1=\chi_2=...=\chi_n$,  and there exists a finite etale cover $\widetilde{Y}\longrightarrow Y$, a finite abelian extension $E\longrightarrow \widetilde{E}$ and a $p$-adic continuous character \[\psi:G^{ab}(\Q)\backslash G^{ab}(\A_f)/\det(\mathcal{U}^p) \longrightarrow O(\widetilde{Y})^*, \] such that for almost all  prime $q$ of $\Q$ that splits in the field $\widetilde{E}$, the element $\psi(\pi_q)$ lies in $O(Y)^*$, and $T_{q,x_q}(f)=deg(T_{q,x_q})\psi(\pi_q)^{-v_q(\det(x_q))}f$. Moreover, the cover $\widetilde{Y}$, the field $\widetilde{E}$ and the map $\psi$ only depends on $Y$.

\end{thm}

\begin{proof}
These two statements are parallel and can be proved by similar ideas.

For the first statement, by definition, for any element $a \in G(\A_f)$ and $g \in \mathcal{U}$, we have $\lambda(a.g)=\lambda(a)g_p$. Pick up an element $g \in G^{der}(A_f)\cap \mathcal{U}$, we know \[\lambda(a)=\lambda(a.g)=\lambda(a).g_p.\] Notice that $g_p$ can be any element in $SL(n,\Q_p) \cap Iw_p$ (just construct $g$ with other components being trivial). Therefore for any $\widetilde{f} \in A_{X,r}$, we have \[(\widetilde{f}.g_{p}^{-1}-\widetilde{f},\lambda(a))=0.\] This property will force $\lambda$ to be zero,

For simplicity, by abuse of notations, still let $\lambda$ denote $\lambda(a)$. Then it is an element in $D_{X,r}$. Recall that we have explicit this module as $\oplus_{k} O(Y)[\langle C_{i,j} \rangle]$, it is enough to show that each component of $\lambda$ is zero, Take a summand $O(Y)[\langle C_{i,j} \rangle]$ and suppose the corresponding part of $\lambda$ is $C$. Through induction, we will show that each $C_{\alpha}$ is zero.

We will do induction for the number $\alpha_{2,1}$. If it is 0, then consider $\widetilde{\alpha}=\alpha+(1,0,\cdots,0)$. The polynomial $Z^{\widetilde{\alpha}}$ certainly lies in the module $O(Y)\langle Z_{i,j}\rangle$ (view this module as the corresponding summand of $A_{X,r}$). Take an element $g_{p}=R_{2,1}(a_{2,1})$ with $a_{2,1} \in \Z_p$ with large enough valuation such that $Z^{\widetilde{\alpha}}.g_p$ still lies in this summand and suppose \[Z^{\widetilde{\alpha}}.g_p=(Z_{2,1}+\delta)\prod_{(i,j)\neq (2,1)}Z^{\alpha_{i,j}} .\]

Then the difference $Z^{\widetilde{\alpha}}.g_p-Z^{\widetilde{\alpha}}$ is exactly $\delta Z^{\alpha}$. By the above property, $(\delta Z^{\alpha}, C)=0$. Then $C_{\alpha}=0$.

Then we proceed via induction. Suppose for any $\alpha$ with $\alpha_{2,1}< N_{0}$ ($N_0$ is a positive integer), we have $C_{\alpha}=0$. Now for an element $\alpha$ with $\alpha_{2,1}=N_0$, we do the above process again to get $\widetilde{\alpha}$, $g_p$ and \[Z^{\widetilde{\alpha}}.g_p=(Z_{2,1}+\delta)^{N_0+1}\prod_{(i,j)\neq (2,1)}Z^{\alpha_{i,j}} .\]

Notice that $(Z_{2,1}+\delta)^{N_0+1}-Z_{2,1}^{N_0+1}=\delta Z_{2,1}^{N_0}+\sum_{e<N_0}\delta_{e}Z_{2,1}^e$, then $C_{\alpha}=0$. By induction, we finished the proof of the first statement.

Now we turn to the second statement.

Still pick up any $g \in G^{der}(\A_f) \cap \mathcal{U}$ and $x \in G(\A_f)$, we get \[f(x)=f(x.g)=f(x).g_p.\] For any $x_0 \in G(\A_f)$ with $f(x_0) \neq 0$ (such $x_0$ exists because $f$ is nonzero), set $\widetilde{f}=f(x_0)$, an element in $A_{X,r}$. Then $\widetilde{f}$ is invariant under the group $SL(n,\Q_p) \cap Iw_p$. Then the $Iw_p$-action factors through the abelian quotient $\det:Iw_p\twoheadrightarrow \Z_p^*$. Denote the kernel of $\det$ as $Iw_p^{der}$. Notice that $Iw_p=Iw_p^{der} T(\Z_p)$, for any $n_0 \in \overline{N}(\Z_p)$ and $b_0 \in T(\Z_p)$, we have \[\widetilde{f}(n_0.b_0)=(\widetilde{f}.n_0)(b_0)=\widetilde{f}(b_0)=\chi(b_0)\widetilde{f}(1).\] In particular, $\widetilde{f}(1) \neq 0$. For any $t \in \Z_p^*$, set $b=diag(t,t^{-1},1,...,1)$, then \[\widetilde{f}(1)=\widetilde{f}(b)=\chi_1(t)\chi_2(t)^{-1}\widetilde{f}(1).\] Because $O(Y)$ is a domain, we get $\chi_1=\chi_2$. Similarly, we find \[\chi_1=\chi_2=...=\chi_n,\] the weight $\chi$ is central.

What's more, the $\det$ map has a (non-canonical) section $s: \Z_p^*\longrightarrow Iw_p$,  sending $a \in \Z_p^*$ to $diag(a,1,...,1)$. Then we can write down the $Iw_p$ action more clearly, for any $g \in Iw_p$, we have \[\widetilde{f}.g=\chi_1(\det(g))\widetilde{f}.\]

Notice that the map $\det: G\longrightarrow G^{ab}$ between two group schemes over $\Q$ also has a non-canonical section via picking up an anisotropic vector inside the Hermitian space over $E$, then the map $G(\Q)\longrightarrow G^{ab}(\Q)$ is surjective, and we have $G^{ab}(\Q) \cap \det(\mathcal{U})=1$ (we can also shrink the level $\mathcal{U}$ in the beginning to get this property).

Consider the $p$-adic continuous character $\psi_0:\det(\mathcal{U})\longrightarrow \det(\mathcal{U}_p)=Z_p^*\longrightarrow O(Y)^*$, since $\det(\mathcal{U}) \cap G^{ab}(\Q)=1$, we  get a $p$-adic character (again denote it by $\psi_0$): \[G^{ab}(\Q)\backslash \det(\mathcal{U}) / \det(\mathcal{U}^p) \longrightarrow O(Y)^*. \] Notice the source is a subgroup of $G^{ab}(\Q) \backslash G^{ab}(\A_f)/ \det(\mathcal{U}^p)$ with finite index (again by generalized finiteness of class group), therefore there exists a finite etale cover $\widetilde{Y}\longrightarrow Y$ and a $p$-adic continuous character \[\psi:G^{ab}(\Q) \backslash G^{ab}(\A_f)/ \det(\mathcal{U}^p) \longrightarrow O(\widetilde{Y})^* \] extending the original character $\psi_0$.

Recall $G^{ab}=U(1)$, now apply class field theory to the field $E$, we can find a finite abelian extension (indeed anticyclotomic) $E\longrightarrow \widetilde{E}$ with a canonical isomorphism \[Gal(\widetilde{E}/E)\cong G^{ab}(\Q) \backslash G^{ab}(\A_f)/\det(\mathcal{U}).\] In particular, for any prime $q$ of $\Q$ that splits  in $\widetilde{E}$, we know that  $\pi_q$ lies in $G^{ab}(\Q)\det(\mathcal{U})$.

Further assume that $q \notin S_0$, let's compute the Hecke action $T_{q,x_q}$ ($x_q \in G(\Q_q)$). First observe that \[\mathcal{U}x_q\mathcal{U}=\coprod_{i}\mathcal{U}x_q y_i, \ y_i \in (x_q^{-1}\mathcal{U}_qx_q \cap \mathcal{U}_q) \backslash \mathcal{U}_q,\] therefore for any $g \in G(\A_f)$, we have \[T_{q,x_q}(f)(g)=\sum_{i}(f|_{x_qy_i})(g)\]\[=\sum_{i}f(gy_i^{-1}x_q^{-1})=\sum_i f(gx_q^{-1}y_{i}^{-1})\] \[=\sum_i ((f|y_i)|x_q)(g)=deg(T_{q,x_q})(f|x_{q})(g).\] Here the last equality follows from  the $\mathcal{U}$-invariance of $f$.

Therefore it remains to compute  the term \[(f|x_{q})(g)=f(gx_q^{-1}).\]

By our assumption on $q$, we know that $\det(\pi_q) \in G^{ab}(\Q)\det(\mathcal{U}) $. On the other hand, $\det(x_q)=\pi_q^{v_q(\det(x_q))}\varepsilon$ with $\varepsilon \in Z_q^*$, while $\det(\mathcal{U}_q)=\Z_q^*$, we find that \[\det(x_q) \in G^{ab}(\Q)\det(\mathcal{U}).\] Then $\psi(\det(x_q))=\psi_0(\det(x_q)) \in O(Y)^*$.

Notice that $\det:G(\Q)\longrightarrow G^{ab}(\Q)$ is surjective, therefore we can find $x_0 \in G(\Q)$ and $x_1 \in \mathcal{U}$ such that $\det(x_q)^{-1}=\det(x_0)\det(x_1)$. Now we have \[f(g x_q^{-1})=f(gx_0x_1)=f(x_0gx_1)\] \[=f(gx_1)=f(g).x_{1,p}=\chi_1(\det(x_{1,p}))f(g),\] here we use the property that $f$ is abelian and $f$ factors through $G(\Q)\backslash G(\A_f)$.

Finally from the construction of $\psi$, we  get \[\chi_1(\det(x_{1,p}))=\psi(\det(x_q))^{-1}=\psi(\pi_q)^{-v_q(\det(x_q))}.\] In summary, \[T_{q,x_q}(f)=deg(T_{q,x_q})\psi(\pi_q)^{-v_q(\det(x_q))}f.\] Moreover, the construction of $\widetilde{E}$, $\widetilde{Y}$ and $\psi$ only depends on $Y$ (independent of $f$). We're done.

\end{proof}

Here for simplicity we work with definite unitary groups over $\Q$, but indeed this abelian Ihara lemma  holds in more general setting.  If $G$ is a reductive group over $\Q$ with $G(\R)$ being compact and the derived subgroup $G^{der}$ is simply connected, then this theorem (assume $G$ split at $p$)  works similarly. This generalization includes all unitary  group over any totally real number field which is definite at all real places.  The theory of overconvergent $p$-adic forms on such general $G$ is studied in \cite{loeffler2011overconvergent}.

In the proof of this theorem, we haven't used too much special properties of unitary groups and some special statements can be replaced by general methods. For instance, we use a special root subgroup $R_{2,1}$, but indeed the only property we need is that this root subgroup commutes with other root subgroups. Through the commutator relations between root subgroups, we can always find such a root (e.g. highest root is suitable for us). With more effort, it maybe possible to further relax the assumption that $G$ is split at $p$.

Another special fact during the proof is that we use certain non-canonical section to show $G(\Q)\twoheadrightarrow G^{ab}(\Q)$. Now we illustrate a  general strategy to deduce this  surjection.

Consider the exact sequence $1\longrightarrow G^{der}\longrightarrow G\longrightarrow G^{ab}\longrightarrow 1$. Because $G^{der}$ is simply connected, for any non-Archimedean field $k$, the Galois cohomology $H^1(k,G^{der})$ vanish (see \cite{bruhat1987groupes} or theorem 6.4 in \cite{platonov1993algebraic}). And the Hasse principle holds for simply connected groups, the map $H^1(\Q,G^{der})\longrightarrow \prod_{v}H^1(\Q_v,G^{der})$ ($v$ runs over all place) is injective (indeed, it is bijective, see theorem 6.6 in \cite{platonov1993algebraic}). We only need to care about $H^1(\R,G^{der})$ now. Because $G^{ab}(\R)$ is a connected compact lie group (indeed isomorphic to products of $U(1)(\R)$), the map $G(\R)\longrightarrow G^{ab}(\R)$ is surjective. Then use the above exact sequence, we get an injection $H^1(\R,G^{der})\hookrightarrow H^1(\R,G)$. Consider the following commutative diagram
\[\xymatrix{
H^1(\Q,G^{der}) \ar@{^{(}->}[d] \ar[r] & H^1(\Q,G) \ar[d] \\
H^1(\R,G^{der}) \ar@{^{(}->}[r] & H^1(\R,G)} \] the map $H^1(\Q,G^{der})\longrightarrow H^1(\Q,G)$ is injective. Therefore the map $G(\Q)\longrightarrow G^{ab}(\Q)$ is surjective in general. And we can similarly construct $\psi$, $\widetilde{Y}$ and $\widetilde{E}$ etc in the later general setting.

\begin{rem}
Regard the second statement, later we will define the notation of \textit{very Eisenstein} (following the convention of \cite{newton2016level}).
\end{rem}

\section{$p$-adic level raising}
\label{sectionlevelraising}

We first introduce a natural pairing and make some duality arguments. Finally combine with the abelian Ihara lemma, we  deduce $p$-adic level raising results.

\subsection{A natural pairing}
\label{sectionpairing}
Let $X\longrightarrow \mathcal{W}$ denote a reduced affinoid.

We define a natural pairing between $\mathcal{S}_X(\mathcal{U},r)$ and $\mathcal{D}_X(\mathcal{U},r)$, which is important to do duality arguments. Take a $\Q$-valued Haar measure $\mu_{\mathcal{U}}$ on $G(\Q)\backslash G(\A_f)$ with $\mu_{\mathcal{U}}(\mathcal{U})=1$. Then we define the pairing as follow: \[(f,\lambda)=\int_{G(\Q)\backslash G(\A_f)} (f(g),\lambda(g))d\mu_{\mathcal{U}},\] where $f \in \mathcal{S}_X(\mathcal{U},r)$ and $\lambda \in \mathcal{D}_X(\mathcal{U},r)$. More explicitly, we can write this integration as a finite sum. Fix a set of double coset representatives $\{x_i|1\leq i\leq h\}$ for $G(\Q)\backslash G(\A_f)/\mathcal{U}$, then we can rewrite the pairing as (use neatness of $\mathcal{U}$) \[(f,\lambda)=\sum_{i=1}^{h}(f(x_i),\lambda(x_i)).\]

Recall that through this set of representatives, we have isomorphisms $\mathcal{S}_X(\mathcal{U},r) \cong A_{X,r}^{h}$ and $\mathcal{D}_X(\mathcal{U},r) \cong D_{X,r}^{h}$, then through this natural pairing, we can identify $\mathcal{D}_{X}(\mathcal{U},r)$ with the $O(X)$-dual of $\mathcal{S}_{X}(\mathcal{U},r)$ (thus justify the name \textit{dual}). Moreover, regarding the slope decomposition (see  section \ref{sectionslope}) \[\mathcal{S}_X(\mathcal{U},r)=\mathcal{S}_X(\mathcal{U},r)^{Q}\oplus Oth,\] we define $\mathcal{D}_X(\mathcal{U},r)^Q$ to be the submodule of $\mathcal{D}_{X}(\mathcal{U},r)$ consists of maps that sends $Oth$ to 0. Then it is naturally the dual module of $\mathcal{S}_X(\mathcal{U},r)^{Q}$.

Next we show that this pairing is well behaved under the  Hecke action.

\begin{prop}
For any $f \in \mathcal{S}_X(\mathcal{U},r)$, $\lambda \in \mathcal{D}_X(\mathcal{V},r)$ and $g \in G(\A_f)$ with $g_p \in Iw_p$, we have $(f|[\mathcal{U}g\mathcal{V}], \lambda)=(f,\lambda|[\mathcal{V}g^{-1}\mathcal{U}])$.

\end{prop}

\begin{proof}

The proof is similar to the proof of proposition 2.10 of \cite{newton2016level}.

Let $Prod$ denote $(f|[UgV], \lambda)$. Take a set of representatives $\{x_i\}$ for $g^{-1}\mathcal{U}g \cap \mathcal{V} \backslash \mathcal{V}$ and a set of representative $\{y_j\}$ for $G(\Q) \backslash G(\A_f)/\mathcal{V}$. By definition, we have \[Prod=\sum_{y_j\in G(\Q) \backslash G(\A_f)/\mathcal{V} }((f|[\mathcal{U}g\mathcal{V}])(y_j),\lambda(y_j))\] \[=\sum_{x_i \in g^{-1}\mathcal{U}g \cap \mathcal{V} \backslash \mathcal{V}}\sum_{y_j\in G(\Q) \backslash G(\A_f)/\mathcal{V}} ((f|(gx_i))(y_j),\lambda(y_j))\] \[=\sum_{x_i \in g^{-1}\mathcal{U}g \cap \mathcal{V} \backslash \mathcal{V}}\sum_{y_j\in G(\Q) \backslash G(\A_f)/\mathcal{V}} (f(y_jx_i^{-1}g^{-1}).(gx_{i})_{p},\lambda(y_j))\] \[=\sum_{x_i \in g^{-1}\mathcal{U}g \cap \mathcal{V} \backslash \mathcal{V}}\sum_{y_j\in G(\Q) \backslash G(\A_f)/\mathcal{V}} (f(y_jx_i^{-1}g^{-1}) ,\lambda(y_j).(x_i^{-1}g^{-1})_p) \] \[=\sum_{x_i \in g^{-1}\mathcal{U}g \cap \mathcal{V} \backslash \mathcal{V}}\sum_{y_j\in G(\Q) \backslash G(\A_f)/\mathcal{V}} (f(y_jx_i^{-1}g^{-1}) ,\lambda(y_jx_{i}^{-1}).g^{-1}_p).\] Here for the last equality, we use the property $\lambda(y_jx_{i}^{-1})=\lambda(y_j).x_{i,p}^{-1}$.

Because the level $\mathcal{V}$ is neat, the multiplication   induces an isomorphism \[G(\Q) \backslash G(\A_f)/\mathcal{V} \times   \mathcal{V}/g^{-1}\mathcal{U}g \cap \mathcal{V} \cong G(\Q)\backslash G(\A_f)/g^{-1}\mathcal{U}g \cap \mathcal{V},\] \[(y_{j},x_{i}^{-1})\mapsto y_j x_i^{-1}. \] This is a routine check: suppose $y_jx_i^{-1}=\delta_0 y_{j_1}x_{i_1}^{-1}\delta_1$, with $\delta_0 \in G(\Q)$ and $\delta_1 \in g^{-1}\mathcal{U}g \cap \mathcal{V}$, then $y_{j}$ and $y_{j_1}$ is the same in the double coset $G(\Q) \backslash G(\A_f)/\mathcal{V}$, thus $y_j=y_{j_1}$. Then $y_{j}^{-1}\delta_0^{-1}y_j=x_{i_1}^{-1}\delta_1 x_i$ lies in $G(\Q) \cap \mathcal{V}$, which is trivial. Then $\delta_0=1$ and $x_i=x_{i_1}$.

Denote $z_{k}=y_j x_i^{-1}$, the set $\{z_k\}$ is a set of representatives for $G(\Q)\backslash G(\A_f)/g^{-1}\mathcal{U}g \cap \mathcal{V}$ and we can rewrite the above equality as \[Prod=\sum_{z_k \in G(\Q)\backslash G(\A_f)/g^{-1}\mathcal{U}g \cap \mathcal{V}}(f(z_kg^{-1}),\lambda(z_k).g_p^{-1}).\] The right multiplication by $g^{-1}$ induces an isomorphism \[G(\Q)\backslash G(\A_f)/g^{-1}\mathcal{U}g \cap \mathcal{V}\cong G(\Q)\backslash G(\A_f)/\mathcal{U} \cap g\mathcal{V}g^{-1}\] \[z_k\mapsto w_k=z_kg^{-1}.\] Then we have \[Prod=\sum_{w_k \in G(\Q)\backslash G(\A_f)/\mathcal{U} \cap g\mathcal{V}g^{-1}}(f(w_k),\lambda(w_kg).g_p^{-1})\] \[=\sum_{w_k \in G(\Q)\backslash G(\A_f)/\mathcal{U} \cap g\mathcal{V}g^{-1}}(f(w_k).g_p,\lambda(w_kg)).\]

Then observe that we can do similar computation for $(f,\lambda|[\mathcal{V}g^{-1}\mathcal{U}])$, it is exactly the above sum. We're done.

\end{proof}

Notice that the $Iw_p$-module $A_{X,r}$ has a further monoid $\mathcal{M}$-action (see section \ref{sectionpadicforms}), if $g_p \in \mathcal{M}$, this proposition still works. Under this pairing, we have identified $\mathcal{D}_X(\mathcal{U},r)$ with $\mathcal{S}_X(\mathcal{U},r)^*$ and this proposition shows that the Hecke action on $\mathcal{D}_X(\mathcal{U},r)$ is exactly the dual action (see section \ref{sectionslope} for notations of Hecke action).

\begin{rem}
Here for simplicity, we work with neat levels. This pairing can be defined without this assumption. And it is still Hecke equivariant. But the computation is more subtle due to  some   volume factors.   In the setting of definite quaternion algebras, see \cite{newton2016level} proposition 2.10 for more details.
\end{rem}

On the other hand, through the slope decomposition, since $\mathcal{S}_X(\mathcal{U},r)^Q$ is a finitely generated projective $O(X)$-module, it is reflexive, i.e. the natural map $\mathcal{S}_X(\mathcal{U},r)^Q\longrightarrow (\mathcal{S}_X(\mathcal{U},r)^Q)^{**}$ is an isomorphism. This is because that any $O(X)$-linear map between finite $O(X)$-modules are automatically continuous. See proposition 2.1 of \cite{buzzard2007eigenvarieties} for more details. In particular, we can compute the continuous dual just by the usual dual in commutative algebra. Then it is an easy  exercise, see lemma 4 of \cite{newton2011geometric} for details. Thus the pairing is perfect if we restrict to $\mathcal{S}_{X}(\mathcal{U},r)^Q \times \mathcal{D}_X(\mathcal{U},r)^Q$. What's more, for any closed immersion $Y\hookrightarrow X$ with $Y$ being a reduced affinoid, recall that we have \[\mathcal{S}_{Y}(\mathcal{U},r)^{Q}\cong \mathcal{S}_X(\mathcal{U},r)^Q \otimes_{O(X)} O(Y),\] combine  with the finitely generated and projective property, we also get the base change property for dual modules: \[\mathcal{D}_{Y}(\mathcal{U},r)^{Q}\cong \mathcal{D}_X(\mathcal{U},r)^Q \otimes_{O(X)} O(Y).\]

\subsection{old and new forms}
\label{oldandnew}
From now on, we assume that $n=3$. Further we  specify three level groups to study level raising questions.

Pick up a level subgroup $\mathcal{U}_0$, recall that  we have defined a finite set $S_0$ of  ('bad') primes. Take an odd prime $l$ that is inert in $E$ and $l \notin S_0$. Then $G(\Q_l)$ is an unramified unitary group $U(3)(\Q_l)$ and it is rank one. We further assume that  the $l$-part $\mathcal{U}_{0,l}$ is a hyperspecial subgroup of $G(\Q_l)$, thus corresponding to a hyperspecial vertex in the Bruhat-Tits tree for $G(\Q_l)$.  Let $T_l$ denote  the standard local Hecke operator at $l$ (generator for this spherical Hecke algebra). Through double  coset operator it also acts on the space of $p$-adic forms or its dual module. Let $\widetilde{\mathcal{U}_l}$ denote the maximal compact open subgroup of $G(\Q_l)$ corresponding to an adjacent vertex. Then $\widetilde{\mathcal{U}_l} \cap \mathcal{U}_{0,l}$  is an Iwahori subgroup. This Bruhat-Tits tree is bi-homogeneous: each hyperspecial vertex has $l^3+1$ adjacent vertices and each special (but not hyperspecial) vertex has $l+1$ adjacent vertices. Each adjacent vertex of a hyperspecial vertex is a special but not hyperspecial vertex and vice versa. In terms of this tree, the Hecke operator $T_l$ sends each hyperspecial vertex into the formal sum of closest hyperspecial vertices. Let $S_1=S_0 \cup \{l\}$, and we get two other level subgroups $U_1=U_{S_0}\times \widetilde{\mathcal{U}_l}\times \mathcal{U}_0^{S_1}$ and $V=U_0 \cap U_1$. These three level subgroups  only differ at the $l$-part. And we let $\mathbb{T}^{S_0}$ denote the integral abstract tame Hecke algebras away from $S_0$. And $\mathbb{T}^{S_1}$ denote similarly the abstract tame Hecke algebras away from $S_1$. And $\mathbb{T}^{S_0}$ is generated by $T_l$ as $\mathbb{T}^{S_1}$-algebra.

Still let $X\longrightarrow \mathcal{W}$ denote a reduced affinoid. To simplify notations, we introduce \[L_0=\mathcal{S}_X(\mathcal{U},r)^{Q},\ L_0^*=\mathcal{D}_X(\mathcal{U},r)^Q,\]
\[L_1=\mathcal{S}_X(\mathcal{U}_1,r)^Q, \  L_1^*=\mathcal{D}_X(\mathcal{U},r)^Q,\] \[M=\mathcal{S}_X(\mathcal{V},r)^Q, \ M^*=\mathcal{D}_X(\mathcal{V},r)^Q.\]

Now we introduce the \textit{level raising} map $i:L_0 \oplus L_1 \longrightarrow M$ by \[i(f_0,f_1)=f_0|[\mathcal{U}_01\mathcal{V}]+f_1|[\mathcal{U}_11\mathcal{V}].\] The double coset operator $[U_01V]$ is easy to describe, it is just the 'forget' map, \[f_0|[\mathcal{U}_01\mathcal{V}]=f_0.\] Similarly $f_1|[\mathcal{U}_11\mathcal{V}]=f_1$. Define the image $im(i)$ inside $M$ to be the space of \textit{old (at $l$) forms}.

We also have a \textit{level lowering}  map $i^+:M\longrightarrow L_0 \oplus L_1$ by \[i^+(f)=(f|[\mathcal{V}1\mathcal{U}_0],f|[\mathcal{V}1\mathcal{U}_1]).\] Define its kernel $ker(i^+)$ to be the space of \textit{new (at $l$) forms}.

The level raising question is to  find common Hecke support for $im(i)$ and $ker(i^+)$. The space of old forms $im(i)$ is relative easier to describe while $ker(i^+)$ is a little mysterious at present. The basic strategy (by Ribet, Taylor, Newton...) is to employ the previous pairing to translate $ker(i^+)$ into  dual side. Next section we will  do such duality arguments.

Use the same double coset operators, we get maps in the dual side: \[j:L_0^* \oplus L_1^* \longrightarrow M^*,\] \[j^+:M^*\longrightarrow L_0^* \oplus L_1^*.\]  In the classical setting, the space of automorphic forms with fixed level and weight is finite dimensional and self dual under the previous pairing, the map $i^+$ is the dual map for $i$. But here the map $j^+$   is the dual map for $i$, and the map $j$ is the dual map for $i^+$:
\[(i(f_0,f_1),\lambda)=((f_0,f_1),j^+(\lambda)),\]  \[(i^+(f),(\lambda_0,\lambda_1))=(f,j(\lambda_0,\lambda_1)).\]

The starting point of level raising is to compute the composition $i^+ \circ i$. It is an endomorphism of $L_0 \oplus L_1$ by a  ('level changing') matrix (acts from the right):
\[\begin{pmatrix}
l^3+1 & [\mathcal{U}_0 1 \mathcal{U}_1] \\
[\mathcal{U}_1 1 U_0] & l+1
\end{pmatrix}.\]

We explain a little about this computation. The computation of these double coset operators is purely local about the place $l$. And each operator can be interpreted by some combinatorial operators about the Bruhat-Tits tree. See Section 3.5 of \cite{bellaiche2006u3} for more details. Through  combinatorial argument of the Bruhat-Tits tree, we can compute this matrix easily. Notice that our matrix is different from their matrix (proposition 3.4.5 of \cite{bellaiche2006u3}), this is because we're using two conjugacy class of maximal open compact subgroups of $G(\Q_l)$ while they only used hyperspecial subgroups. More concretely, their level raising map (section 5.3.6 of \cite{bellaiche2006u3}) is for $L_0 \oplus L_0 \longrightarrow M$ by the matrix
 \[
 \begin{pmatrix}
 1  \\
  [\mathcal{U}_0 1\mathcal{U}_1]\cdot [\mathcal{U}_1 1 \mathcal{V}]-[\mathcal{U}_01\mathcal{V}]
 \end{pmatrix}.
 \] In the last section (see section \ref{localpicture}), we explain the  reason why our definition (use $L_0 \oplus L_1$ instead of $L_0\oplus L_0$) is better.

What's more, through the help of lemma 3.5.3 (1) of \cite{bellaiche2006u3} (or compute it directly), we find the composition of the following map \[L_0 \xrightarrow{[\mathcal{U}_0 1 \mathcal{U}_1]}L_1 \xrightarrow{[\mathcal{U}_11\mathcal{U}_0]}L_0\] is exactly \[T_l+(l^3+1).\]

Apply these computation to the  dual side, obviously we get the same matrix for $j^+ \circ j$. And the  composition of analogous above maps is \textbf{still} $T_l+(l^3+1)$. According to our notation, a single local Hecke operator will act on $S_{X}(U_0,r)$ and $D_X(U_0,r)$ through two different double operators, their representative elements are inverse to each other. But in this situation, these two double coset are indeed the same. This can also be obtained from the uniqueness  of the generator for the spherical Hecke algebra.

In \cite{newton2016level}, Newton observed a useful lemma (see lemma 2.12 in that paper) in commutative algebra, which will help us to compute dimension of certain locus of eigenvarieties. We briefly recall his lemma:

\emph{Let $R_0$ denote a Noetherian domain which is normal and equidimensional of dimension $d$. Suppose $R_1$ is an $R_0$-algebra which is integral over $R_0$ and torsion free. Then $R_1$ is equidimensional of dimension $d$.}

Now we can prove the following result, which is crucial to do duality arguments in later sections.

\begin{prop}
\label{injection}
If $X\hookrightarrow \mathcal{W}$ is an irreducible reduced affinoid which is admissible open in $\mathcal{W}$, then the map $i^+ \circ i$ is injective.
\end{prop}
\begin{proof}

We adapt Newton's proof (see proposition 2.13 of \cite{newton2016level}), replacing the application of the generalized Petersson-Ramanujan conjecture with two key tools: (1) our abelian Ihara lemma (Theorem \ref{ihara}), and (2) the semisimplicity of classical automorphic representations.

Let $\widetilde{L}$ denote the $O(X)$-module $ker(i^+ \circ i)$ and $L$ denote its image inside $L_0$ under the first projection. Let $\mathbf{H}_{L_0}$ denote the image of $\mathbb{T}^{S_0}\otimes \mathcal{H}_p^{-} \otimes O(X)$ in $End_{O(X)}(L_0)$.  Notice that the Hecke algebra $\mathbb{T}^{S_1}\otimes \mathcal{H}_p^{-}$ acts on $L$. At present we don't know whether $T_l$ also stables $L$, but we will verify this quickly.

If $\widetilde{L} \neq 0$, for any $(f_0,f_1)\in \widetilde{L}$, we have \[(l^3+1)f_0+f_1|[\mathcal{U}_11\mathcal{U}_0]=0,\] \[f_0|[\mathcal{U}_01\mathcal{U}_1]+(l+1)f_1=0.\] Combine them together, we get \[(T_l-l(l^3+1))f_0=0.\] In particular, the Hecke  operator $T_l$ acts  on $L$ through the scalar $l(l^3+1)$. Then let $\mathbf{H}_{L}$ denote the image of $\mathbb{T}^{S_0}\otimes \mathcal{H}_p^{-} \otimes O(X)$ in $End_{O(X)}(L)$, through the restriction from $L_0$ to $L$, we get a surjection \[\mathbf{H}_{L_0}\twoheadrightarrow \mathbf{H}_{L}.\]   On the other hand, because $L_0$ is a torsion free $O(X)$-module, so does $L$, $\mathbf{H}_{L}$ and $(\mathbf{H}_{L})^{red}$. Then use Newton's lemma, we find that $\mathbf{H}_{L}^{red}$ is equidimensional of dimension $O(X)$ (which is $n$). Then we have a closed immersion between reduced rigid spaces \[Sp(\mathbf{H}_{L}^{red})\hookrightarrow Sp(\mathbf{H}_{L_0}^{red})\] with the same dimension and are all equidimensional. Then its image is a finite union of irreducible components. Apply the eigenvariety machine, the zero locus of the Hecke operator $T_{l}-l(l^3+1)$ will contain an irreducible component of the eigenvariety $\mathcal{E}(\mathcal{U}_0)$. Next we will show this is impossible.

For any classical point $z_0$ inside this irreducible component, let $x_0$ denote the classical (locally algebraic) weight. Let level subgroups vary, and the resulting whole space of automorphic forms is canonically isomorphic to a space of automorphic representations of $G(\A_f)$, which is semisimple. See \cite{loeffler2011overconvergent} proposition 3.8.1 and theorem 3.9.2 for more details. In particular, let $f$ denote a $p$-adic Hecke eigenform inside $\mathcal{S}_{x_0}(\mathcal{U},r)^Q$ corresponding to this classical point $z_0$. Through the above identification of two whole spaces, we can realize $f$ canonically as a vector inside the automorphic representation $\oplus_{k}\Pi_{k}$ (each $\Pi_k$ is an irreducible representation of $G(\A_f)$ and this is a finite sum). And  each $\Pi_k$ further decompose as $\Pi_{k,l}\otimes \widetilde{\Pi_{k,l}}$, where $\Pi_{k,l}$ is an irreducible representation of $G(\Q_l)$ and $\widetilde{\Pi_{k,l}}$ is an irreducible representation of $G(\A_{f}^{l})$. Because $f$ is invariant under the right multiplication by $\mathcal{U}_{0,l}$. The representation $\Pi_{k,l}$ is unramified with respect to the hyperspecial subgroup $\mathcal{U}_{0,l}$. Such an representation is determined by its Satake parameter and the corresponding full principal series is also classified clearly. See section 3.6 and section 3.7 of \cite{bellaiche2006u3} for a list of such results. They denote the Satake parameter by a number $\alpha$ (the group $G(\Q_l)$ is rank one). The previous discussion shows that the eigenvalue of $T_l$ is $l(l^3+1)$, and the corresponding $\alpha$ is $l^{\pm2}$ by lemma 3.7.1 of \cite{bellaiche2006u3}. However, such a Satake parameter is degenerate and the corresponding full principal series is reducible, with one Jordan Holder factor being a character and another one Jordan Holder factor being the Steinberg representation. Thus the unramified representation $\Pi_{k,l}$ is a \textbf{character}. In particular, $f$ is invariant under the $G^{der}(\Q_l)$ action.

  View $f$ as a continuous function on $G(\A_f)$ and it is invariant under left multiplication by $G^{der}(\Q)$ and right multiplication by $G^{der}(\Q_l)$. Through the strong approximation theorem for $G^{der}(\A_f)$ ($G^{der}$ is simply connected), $G^{der}(\Q)G^{der}(\Q_l)$ is dense in $G^{der}(\A_f)$, then the form $f$ is invariant under multiplication by $G^{der}(\A_f)$ and thus it is an abelian form. Now we can apply the previous abelian Ihara lemma to $f$. In particular, the abelian property forces the weight $x_0$ to be \textbf{central}.

However, such central weights only occupy a proper (one dimensional) Zariski closed subspace of the (three dimensional) whole weight space $\mathcal{W}$. The complement is a Zariski open subspace of $\mathcal{W}$ with plenty of classical points. And we can apply the Zariski density of classical points (see corollary 3.13.3 of \cite{loeffler2011overconvergent}) to get many classical points in this irreducible component of $\mathcal{E}(\mathcal{U}_0)$ with non-central weights, which is a contradiction.

Therefore the kernel of $i^+ \circ i$ is zero, we're done.

\end{proof}

This proposition shows that an irreducible components of the eigenvariety can't be both new and old, which is fine. But for a point, it maybe both new and old. And level raising questions is about such phenomenon. Moreover, the above proof shows that such an exotic point is with a non-central weight is  not classical.

\subsection{some duality results}
\label{duality}
In this section, we assume the reduced irreducible $X$ is an admissible open affinoid in $\mathcal{W}$. Let $F(X)$ denote the fraction field of $O(X)$. For the $O(X)$-module $L_0$, for simplicity we will use $L_{0,F(X)}$ to denote $L_0 \otimes_{O(X)}F(X)$ and similarly for other $O(X)$-modules.

First, we remark that the injectivity of $i^{+}\circ i$ implies the injectivity of $j^+  \circ j$: Notice that each $O(X)$-module $L_0$, $L_1$ and their dual are torsion free. The $F(X)$-vector space $(L_0\oplus L_1)_{F(X)}$ is finite dimensional, thus the injection $i^{+} \circ i$ is indeed an isomorphism on this vector space. Then the dual map $j^  \circ j$ is also injective on the dual vector space. Because $L_0 ^* \oplus L_1 ^*$ is torsion free, thus $j^+ \circ j$ is an injective endomorphism. In particular, the map $i$ and $j$ are injective.

Similar to Newton's ideas, we introduce some auxiliary modules and certain (perfect) pairing. We define the following two chains of modules:
\[\gamma_0=L_0 \oplus L_1,\  \widetilde{\gamma_0}=L_0 ^* \oplus L_1 ^*; \] \[\gamma_1=i^+(M),\ \widetilde{\gamma_1}=j^+(M^*);\]
\[\gamma_2=i^+(M \cap i(\gamma_{0,F(X)})),\ \widetilde{\gamma_2}=j^+(M^* \cap j(\widetilde{\gamma_0}_{F(X)}));\]
\[\gamma_3=i^{+}i(\gamma_0), \ \widetilde{\gamma_3}=j^+ j(\widetilde{\gamma_0}).\]

We have $\gamma_3 \subset \gamma_2 \subset \gamma_1 \subset \gamma_0$ and \[\gamma_2/\gamma_3=i^+(\frac{M\cap i(\gamma_0)_{F(X)}}{i(\gamma_0)})=i^+(M/(i(\gamma_0))^{tors}),\] and similar conclusions for the dual side. In next section, we will see that the torsion module $\gamma_2/\gamma_3$ is closely related with the abelian Ihara lemma.

Start from the previous perfect pairing on $\gamma_0 \times \widetilde{\gamma_0}$, we further produce some other perfect pairing. Each pairing will be equivariant respect to the tame Hecke $\mathbb{T}^{S_1}$ action.

Combine with perfect pairing $\gamma_{0,F(X)} \times \widetilde{\gamma_{0}}_{F(X)}\longrightarrow F(X)$ and the injection $j$ (thus identify $\widetilde{\gamma_{0}}_{F(X)}$  and $j(\widetilde{\gamma_{0}}_{F(X)})$), we get a pairing \[\gamma_0 \times (M^* \cap j(\widetilde{\gamma_{0}}_{F(X)})) \longrightarrow F(X)/O(X),\] and further the following pairing \[P_1:\gamma_0/\gamma_1 \times \frac{M^* \cap j(\widetilde{\gamma_0}_{F(X)})}{j(\widetilde{\gamma_0})}\longrightarrow F(X)/O(X).\] This pairing is perfect due to lemma 6 of \cite{newton2011geometric}. We quickly recall it here:

\textit{The  pairing on} $\gamma_{0,F(X)} \times \widetilde{\gamma_{0}}_{F(X)}$ \textit{induces natural isomorphisms} \[(\gamma_1)^* \cong M^* \cap j(\widetilde{\gamma_{0}}_{F(X)})\ and \ (\gamma_0)^* \cong j(\widetilde{\gamma_{0}}).\]

Similarly we get the second pairing: \[P_2: \frac{M\cap i(\gamma_{0,F(X)})}{i(\gamma_0)}\times \widetilde{\gamma_0}/\widetilde{\gamma_1}\longrightarrow F(X)/O(X).\] Finally combine with the perfect pairing on $M \times M^*$, we get the third  pairing \[P_3:ker(i^+) \times \frac{M^*}{M^*\cap j(\widetilde{\gamma_0}_{F(X)})}\longrightarrow O(X),\] which identify $ker(i^+)$ with the  $O(X)$-dual module of $\frac{M^*}{M^*\cap j(\widetilde{\gamma_0}_{F(X)})}$.

We refer to section 2.8 of \cite{newton2011geometric}  for more details. Although Newton worked with definite quaternion algebras, his proof holds in general.

Here we make some remark. In next section we will introduce the notation of very Eisenstein modules. The abelian Ihara lemma will imply that torsion module like $(M/i(\gamma_0))^{tors}$ is very Eisenstein. In particular, the quotient $\gamma_2/\gamma_3$ is also very Eisenstein. Apply the the pairing $P_1$, this nice duality shows that $\gamma_0/\gamma_1$ is also very Eisenstein. So up to such very Eisenstein modules, the module $\gamma_0/\gamma_3$ is 'close' to the quotient $\gamma_1/\gamma_2$. The duality by pairing $P_3$ will relate $ker(i^+)$ with the module $\widetilde{\gamma_1}/\widetilde{\gamma_2}$. So the study of $\gamma_0/\gamma_3$ provides a bridge connect old forms $im(i)$ with new forms $ker(i^+)$. The interested readers can find the archetype of such ideas in the classical study of \textit{congruence modules}. For example see \cite{ribet1983congruence} and \cite{taylor1989galois} for more details.

\subsection{very Eisenstein modules}

Following the notation of \cite{newton2016level} (see section 2.11 of that paper), we introduce the concept of \textit{very Eisenstein}. The definition is inspired by the abelian Ihara lemma.

As we have specified to $n=3$. We can make the local Hecke algebra more explicitly and compute the related $deg$ function for Hecke operators. For any prime $q$ that splits in $E$ and $q \notin S_0$, the group $G(\Q_q)$ splits as $GL(3,\Q_q)$ and $\mathcal{U}_q$ is conjugated to $GL(3,\Z_q)$. For  the standard spherical Hecke algebra  $H(GL(3,\Q_q),GL(3,\Z_q))$ with respect to $GL(3,\Z_q)$, there are three distinguished Hecke operators $T_{q,i}(1\leq i \leq 3)$, where $T_{q,1}$ corresponds to the double coset $GL(3,\Z_q)diag(q,1,1)GL(3,\Z_q)$ and similarly for $T_{q,2}$ and $T_{q,3}$. Then for any coefficient ring $R_0$ containing $\Q(q^{\pm\frac{1}{2}})$, through the Satake isomorphism, the spherical Hecke algebra is isomorphic to $R_0[T_{q,1},T_{q,2},T_{q,3}^{\pm1}]$. So in most of time, we can restrict to these three Hecke operators to describe this local Hecke algebra. Through the conjugation, $\mathcal{U}_q$ is isomorphic to $GL(3,\Z_q)$. The Hecke algebra $\mathcal{H}_q$ respect to $\mathcal{U}_q$ is thus isomorphic to the standard spherical Hecke algebra. We still use $T_{q,i}$ to denote the corresponding elements. And easy to see that $deg(T_{q,1})=deg(T_{q,2})=\frac{q^3-1}{q-1}=1+q+q^2$ and $deg(T_{q,3})=1$.

Still use previous notations like $X$ and $M$ etc. Let $\mathbf{H}$ denote the image of $O(X)\otimes \mathbb{T}^{S_1}$ in $End_{O(X)}(M)$. Let $M_0$ denote an $\mathbf{H}$-module which is finitely generated over $O(X)$. We call $M_0$ is \textit{very Eisenstein} if for  each prime ideal $\mathfrak{b}$ of $\mathbf{H}$ in the support of $M_0$ with $\mathfrak{p}= \mathfrak{b}\cap  O(X)$, we have:

$\bullet$ the induced weight $T(\Z_p)\longrightarrow (O(X)/\mathfrak{p})^*$ is central.

$\bullet$ there exists a finite etale map $O(X)/\mathfrak{p}\longrightarrow R$, a finite abelian extension $E\longrightarrow \widetilde{E}$ and a $p$-adic continuous character $\psi: G^{ab}(\Q) \backslash G^{ab}(A_f)/\det(\mathcal{U}^p)\longrightarrow R^*$ such that if $q \notin S_0$ is a prime that splits in $\widetilde{E}$, then $\psi(\pi_q) \in (O(X)/\mathfrak{p})^*$ and inside $\mathbf{H}/\mathfrak{b}$, we have \[T_{q,1}-\frac{q^3-1}{q-1}\psi(\pi_q)^{-1}=0,\] \[T_{q,2}-\frac{q^3-1}{q-1}\psi(\pi_q)^{-2}=0,\] \[T_{q,3}-\psi(\pi_q)^{-3}=0.\]




With some basic properties of very  Eisenstein modules (see lemma 2.21 of \cite{newton2016level}) and the abelian Ihara lemma, we can conclude some lemmas regarding modules in the previous section.

\begin{lem}
Let $Y\hookrightarrow X$ be a closed, reduced and irreducible sub-affinoid. Then the module $Tor_1^{O(X)}(M/i(\gamma_0),O(Y))$ is very Eisenstein and the module $Tor_1^{O(X)}(M^*/j(\widetilde{\gamma_0}), O(Y))$ is 0.
\end{lem}

\begin{proof}

The proof is similar to lemma 2.22 and lemma 2.23 of \cite{newton2016level}.

As $i$ is injective, we have an  exact sequence \[0\longrightarrow\gamma_0 \longrightarrow M \longrightarrow M/i(\gamma_0)\longrightarrow0.\]
Because $\gamma_0$ and $M$ are finitely generated projective $O(X)$-modules (thus flat), and apply the functor $(-)\otimes_{O(X)}O(Y)$, we get an exact sequence \[0 \longrightarrow Tor_1^{O(X)}(M/\gamma_0,O(Y))\longrightarrow \gamma_0 \otimes_{O(X)}O(Y)\xrightarrow{i_Y} M \otimes_{O(X)}O(Y).\]

Identify $\gamma_0 \otimes_{O(X)}O(Y)$ with $\mathcal{S}_{Y}(\mathcal{U}_0,r)^{Q}\oplus \mathcal{S}_{Y}(\mathcal{U}_1,r)^{Q}$ and $M \otimes_{O(X)}O(Y)$ with $\mathcal{S}_{Y}(\mathcal{V},r)$, the above map is exactly the level raising map $i_Y$ for weight $Y$. It is enough to show that its kernel is very Eisenstein.

Suppose $i_Y(f_0,f_1)=0$. Then both $f_0$ and $f_1$ are invariant under the right multiplication by $\mathcal{U}_{0,l}$ and $\mathcal{U}_{1,l}$. Notice that $\mathcal{U}_{0,l}$ and $\mathcal{U}_{1,l}$ will generate a subgroup of $G(\Q_l)$ containing $G^{der}(\Q_l)$. Thus $f_0$ and $f_1$ are invariant under right multiplication by $G^{der}(\Q_l)$. Then use the same argument in the proof of proposition \ref{injection}, the forms $f_0$ and $f_1$ are invariant under multiplication by $G^{der}(\A_f)$, now we can the abelian Ihara lemma to them and get the  desired result for $Tor_1^{O(X)}(M/i(\gamma_0),O(Y))$. Apply these arguments in the dual side, we find $Tor_1^{O(X)}(M^*/j(\widetilde{\gamma_0}), O(Y))=0$.

\end{proof}

With some arguments about commutative algebra, we can upgrade this lemma into the following form:

\begin{lem}
(1) The module $(M/i(\gamma_0))^{tors}$ is very Eisenstein.

(2) The module $(M^*/j(\widetilde{\gamma_0}))^{tors}$ is 0.
\end{lem}

We refer to lemma 2.24 of \cite{newton2016level} for the proof. The main idea is to apply some general results about support in commutative algebra. The property of very Eisenstein can be checked via minimal elements in the support. These elements are the same as the minimal elements inside the associated primes. Then it is enough to show that modules like $Tor_1^{O(X)}(M/i(\gamma_0), O(X)/(\alpha))$ ($\alpha$ is nonzero) are very Eisenstein. Apply ideas of devissage, we can pass to consider modules like  $Tor_1^{O(X)}(M/i(\gamma_0), O(X)/\mathfrak{p})$, where $\mathfrak{p}$ is prime ideal. Then this is just the previous lemma. The second statement is proved similarly.

With the help of these lemmas and duality pairings in the previous section, we deduce the following proposition (which is claimed in that section):

\begin{prop}

The modules $\gamma_2/\gamma_3$ and $\widetilde{\gamma_0}/\widetilde{\gamma_1}$ are very Eisenstein. The modules $\gamma_0/\gamma_1$ and $\widetilde{\gamma_2}/\widetilde{\gamma_3}$ are 0.
\end{prop}

\subsection{raise the level}

Now we can deduce some level raising results.

We define a prime ideal $\mathbf{p}$ of $\mathbf{H}$ to be \textit{very Eisenstein} if the $\mathbf{H}$ module $\mathbf{H/p}$ is very Eisenstein. In particular, for a Hecke eigenform, its system of Hecke eigenvalue will produce a maximal ideal of $\mathbf{H}$, if this ideal is very Eisenstein, we also call such a Hecke eigenform \textit{very Eisenstein}.

We remark that as the name suggests, the property of being very Eisenstein is much more strict than being merely Eisenstein (not cuspidal). For example, the central weight condition cuts off lots of representations, e.g. some Hecke eigenforms  coming from the endoscopy via $U(2)\times U(1)$ (thus not cuspidal) may not be very Eisenstein.

Let $\mathbf{H}_0$ denote the image of $\mathbb{T}^{S_0}\otimes O(X)$ inside $End_{O(X)}(L_0)$. Through the embedding $L_0\hookrightarrow L_0 \oplus L_1 \hookrightarrow M$ (the first map is just the natural inclusion), we get a finite map  (via restriction) $\mathbf{H}\longrightarrow \mathbf{H}_0$. For any ideal $I$ of $\mathbf{H}_0$, let $I_M$ denote the inverse image of $I$ in $\mathbf{H}$. For any finite $\mathbf{H}_0$-module $M_0$, if a prime ideal $\mathbf{p}$ lies in $supp_{\mathbf{H}_0}(M_0)$, then $\mathbf{p}_{M}$ also lies in $supp_{\mathbf{H}}(M_0)$.

We have the following proposition concerning $p$-adic level raising:

\begin{prop}
\label{propraising}
Suppose $\mathbf{p}$ is a prime ideal of $\mathbf{H}_0$ such that $\mathbf{p}_M$ is not very Eisenstein. If we further assume that $\mathbf{p}$ contains $T_l-l(l^3+1)$, then $\mathbf{p}_M$ lies in the support of $\mathbf{H}$-module $ker(i^+)$.
\end{prop}

\begin{proof}

Consider the $\mathbf{H}$-module $Q=\widetilde{\gamma_0}/\widetilde{\gamma_3}$. Recall that the composition $j^{+}j$ can be represented by a matrix (acts from right)
\[\begin{pmatrix}
l^3+1 & [\mathcal{U}_0 1 \mathcal{U}_1] \\
[\mathcal{U}_1 1 U_0] & l+1
\end{pmatrix}.\]

Because the $\mathbf{H}_0$ support of $L_0$ and $\widetilde{L_0}$ are equal, then $\mathbf{p}_M$ also lies in the support of $L_0^*$ (over $\mathbf{H}$). So it further lies in $supp_{\mathbf{H}}(\widetilde{\gamma_0})$. If $\mathbf{p}_M \notin supp_{\mathbf{H}}(Q)$, then after localizing to $\mathbf{H}_{\mathbf{p}_M}$, the map $j^+ j$ is a surjective endomorphism for $\widetilde{\gamma_0}_{\mathbf{p}_M}$.

Consider the following $\mathbf{H}$-equivariant surjection \[\rho: \widetilde{\gamma_0}  \twoheadrightarrow L_0^*,\] \[(f_0,f_1)\mapsto -(l+1)f_0+ f_1|[\mathcal{U}_11\mathcal{U}_0].\] After localization, the following composition is  still surjective:
 \[ \widetilde{\gamma_{0}}_{\mathbf{p}_M} \twoheadrightarrow \widetilde{\gamma_{0}}_{\mathbf{p}_M}\twoheadrightarrow L^*_{0,\mathbf{p}_M}.\]

However, this composition is represented by the matrix $\begin{pmatrix}
T_{l}-l(l^3+1) \\
0
\end{pmatrix}$. In particular, it will  imply that $T_{l}-l(l^3+1)$ is an surjective endomorphism for $L^*_{0,\mathbf{p}_M}$. Now we consider the $\mathbf{H}_0$-action on $L_0^*$ and further localize to $\mathbf{p}$, we find that the map $T_{l}-l(l^3+1)$ is a surjective endomorphism for a finitely generated nonzero module $L^*_{0,\mathbf{p}}$. But $\mathbf{p}$ contains $T_l-l(l^3+1)$, such multiplication can't be surjective due to Nakayama's lemma. Therefore $\mathbf{p}_M \in supp_{\mathbf{H}}(Q)$.

The  final proposition of the previous section implies that any prime ideal lying in the support of $\widetilde{\gamma_0}/\widetilde{\gamma_1}$ or $\widetilde{\gamma_2}/\widetilde{\gamma_3}$ is very Einsenstein. Therefore $\mathbf{p}_M$ must lie in the support of $\widetilde{\gamma_1}/\widetilde{\gamma_2}$. Then $\mathbf{p}_M$ also lies in the support of $\frac{M^*}{M^*\cap j(\widetilde{\gamma_0}_{F(X)})}$.

Through the pairing $P_3$, we conclude that $\mathbf{p}_M$  lies in the support of $ker(i^+)$.

\end{proof}

\begin{rem}
The condition that $\mathbf{p}$ contains $T_l-l(l^3+1)$ can be seen as the level raising condition. It is a kind of  $p$-adic analogue of classical level raising condition  in \cite{bellaiche2006u3}.
\end{rem}

\section{Applications}
\label{sectioneigenvariety}

\subsection{intersection points on the eigenvariety} Now we use the $p$-adic level raising results to get some intersection points on the eigenvariety. We will always use the reduced eigenvariety.

Let $\mathcal{E}(\mathcal{V})$ denote the (reduced) eigenvariety with level $\mathcal{V}$, constructed via the Hecke algebra $\mathbb{T}^{S_1}\otimes \mathcal{H}_p^{-}$.  Similarly use Hecke algebra $T^{S_0}\otimes \mathcal{H}_p^{-}$ to construct the (reduced) eigenvariety  $\mathcal{E}(\mathcal{U}_0)$ with level $\mathcal{U}_0$.

For each admissible open (reduced) affinoid $X\hookrightarrow \mathcal{W}$, let $\mathbf{T}$ denote the image of Hecke algebras $\mathbb{T}^{S_1}\otimes \mathcal{H}_p^{-} \otimes O(X)$ inside $End_{O(X)}(\mathcal{S}_X(\mathcal{V},r)^{Q})$. Recall  the space of old forms $im(i)$ inside $\mathcal{S}_X(\mathcal{V},r)^{Q}$, let $\mathbf{T}^{old}$ denote the image of Hecke algebras $\mathbb{T}^{S_1}\otimes \mathcal{H}_p^{-} \otimes O(X)$ inside $End_{O(X)}(im(i))$. Through restriction, we get a natural closed immersion between rigid spaces $Sp(T^{old,red})\hookrightarrow Sp(T^{red})$. For any $Y\hookrightarrow X$ an admissible open sub-affinoid, the module $im(i)$ satisfies base change property. Then we can glue these closed immersions and thus get a closed subspace \[\mathcal{E(V)}^{old}\hookrightarrow \mathcal{E(V)}.\] We refer to lemma 14 of \cite{newton2011geometric} or proposition 4.2 of \cite{newton2016level} for more details. Because $im(i)$ is torsion free over $O(X)$, the same argument during the proof of proposition \ref{injection} shows that $\mathcal{E(V)}^{old}$ is equidimensional and further a union of irreducible components of $\mathcal{E(V)}$. Roughly speaking, it is the Zariski closure of classical points corresponding to old forms. Therefore we may call it \textit{old component}. Similarly, let $\mathbf{T}^{new}$ denote the image of Hecke algebras $\mathbb{T}^{S_1}\otimes \mathcal{H}_p^{-} \otimes O(X)$ inside $ker(i^+)$, the same process ($ker(i^+)$ also satisfies base change property between admissible open affinoids) produces a closed subspace \[\mathcal{E(V)}^{new}\hookrightarrow \mathcal{E(V)}.\] Again the torsion freeness guarantees that it is a union of irreducible components. It is the Zariski closure of classical points corresponding to new forms. We may call it \textit{new component}. The proposition \ref{injection} (injectivity of $i^+ i$) shows that $\mathcal{E(V)}^{new}$ and $\mathcal{E(V)}^{new}$ can't have a common irreducible component. Moreover, as a classical form with level $\mathcal{V}$ is either old at $l$ or new at $l$, apply the Zariski density of classical points, we divide irreducible components of $\mathcal{E(V)}$ into two types.

Through the inclusion $\mathcal{S}_X(\mathcal{U}_0,r)\xrightarrow{[\mathcal{U}_01\mathcal{V}]}\mathcal{S}_X(\mathcal{V},r)$ we have a natural finite map \[\mathcal{E}(\mathcal{U}_0)\longrightarrow \mathcal{E(V)}^{old}.\] Indeed this should be a closed immersion with image being a union of irreducible components. As we can apply the above process to the smaller submodule $\mathcal{S}_X(\mathcal{U}_0,r)$ (as submodule of $\mathcal{S}_X(\mathcal{V},r)$), the resulting rigid space is exactly the image of this finite map. Because this rigid space only omit information of a single Hecke operator $T_l$ inside the usual Hecke algebra for the eigenvariety $\mathcal{E}(\mathcal{U}_0)$, it shouldn't influence too much in the geometry of the eigenvariety. For example, if  we  have certain multiplicity one results, the eigenvalue of  $T_l$ should be determined from other Hecke operators already (although not explicitly). For any point $x_0$ of $\mathcal{S}_X(\mathcal{U}_0,r)$, we say it is \textit{very Eisenstein} if the corresponding maximal  ideal of the Hecke algebra for its image inside $\mathcal{E(V)}$ is very Eisenstein.

Now we can state the main theorem of this paper.

\begin{thm}
\label{main}
Suppose we have a point $\phi$ on $\mathcal{E}(\mathcal{U}_0)$ which is not very Eisenstein and satisfies $T_l(\phi)=l(l^3+1)$. Then the corresponding point inside $\mathcal{E(V)}^{old}$ will also lie in $\mathcal{E(V)}^{new}$.
\end{thm}
\begin{proof}
Let $\widetilde{\phi}$ denote the resulting point in $\mathcal{E(V)}^{old}$. From the construction of eigenvarieties, there exists an admissible open affinoid $X\hookrightarrow W$ such that the point $\phi$ corresponds to  a maximal ideal $\mathfrak{R}_{\phi}$ of $\mathbb{T}^{S_0}\otimes \mathcal{H}_p^{-}\otimes O(X)$ that lies in the support of $\mathcal{S}_X(\mathcal{U}_0,r)^{Q}$. Now we can apply the proposition \ref{propraising} to raise the level. The resulting maximal ideal $\mathfrak{R}_{\widetilde{\phi}}$ for the Hecke algebra $\mathbb{T}^{S_1}\otimes \mathcal{H}_p^{-}\otimes O(X)$ lies in the support of $ker(i^+)$. Therefore $\widetilde{\phi}$ also lies in $\mathcal{E(V)}^{new}$.
\end{proof}

\begin{rem}

Suppose the vanishing locus of the Hecke operator $T_{l}-l(l^3+1)$ is non-empty, then it is a codimension one (thus dimension two) space inside $\mathcal{E}(\mathcal{U}_0)$. During the proof of the proposition \ref{injection}, each classical point inside this vanishing locus is very Eisenstein (in particular has central weight). After cutting out the (at most one dimensional) closed subspace over central weights inside  this locus, the resulting space is still non-empty and two dimensional. This space consists of non-classical points and lies in the intersection of old components and new components by this theorem.

\end{rem}

\subsection{construction of such points and further development}

Now we discuss some further ideas about such points.

As we wrote in the remark, the key step is to find suitable $\mathcal{E}(\mathcal{U}_0)$  such that the vanishing locus of $T_{l}-l(l^3+1)$ is non-empty. In \cite{newton2016level} (see section 4.2 of that paper), James Newton constructed similar intersection points on the eigenvariety for definite quaternion algebra $D$ over $\Q$ via some explicit computations (via Sage). See proposition 4.9 of that paper. His idea is to work with the usual $GL(2)$ and transfer to $D$. Then he turned to the usual $p$-adic ordinary modular forms. He constructed certain nice Hida family and computed the ordinary Hecke algebra explicitly. Then he constructed such intersection points satisfying the level raising condition. Indeed that point lies in the intersection between two Hida families. He further mentioned that example 5.3.2 of \cite{emerton2006variation} is also an example of such intersection points between two Hida families. Although there are many development about generalizations of usual Hida theory to definite unitary groups $U(n)$, such explicit computation is still more difficult than $GL(2)$ cases. For example, in the case of (Hilbert) modular forms, there is an explicit duality between ordinary Hecke algebra and the space of ordinary cusp forms via Fourier coefficients. But such nice result doesn't exist for definite unitary groups ($n>2$). And many other problems make the study of ordinary Hecke algebra for definite unitary groups much harder.

Instead we can try to apply $p$-adic Langlands functoriality to construct such points. As $GL(2)$ is  also closely related with $U(2)$, we can first transfer James Newton's result to the eigenvariety of definite $U(2)$. The methods of  this paper obviously apply to $U(2)$ setting. Then the resulting points on this eigenvariety should relate to degenerate Satake parameter for $U(2)(\Q_l)$. The Bruhat-Tits tree for $U(2)(\Q_l)$ is a homogeneous tree (like the $GL(2,\Q_l)$ case) and the degenerate Satake parameter is $l^{\pm1}$. If we have the desired $p$-adic symmetric square functoriality, then under such functoriality map, the resulting  point on $\mathcal{E}(\mathcal{U}_0)$ for $G=U(3)$ is a desired point satisfying the condition in theorem \ref{main}. The symmetric power functoriality for classical forms is known due to \cite{nt2021symmetric1} and \cite{nt2021symmetric2}. But as our intersection point is non-classical. We can't apply their results directly. We need to do $p$-adic interpolation of such symmetric square functoriality to get a map from the eigenvariety of $U(2)$ to the eigenvariety of $U(3)$. It is reasonable to expect such a map. I'm trying to apply the method of David Hansen (see \cite{hansen2017universal}) to get such $p$-adic symmetric square functoriality. Indeed, Hansen already showed a kind of symmetric square functoriality between eigenvariety of $GL(2)$ and $GL(3)$ (see section 5.4 of that paper), which is very close to our setting. I hope to finish these details later.

What's more, I'm also considering using the method of \cite{hansen2017universal} to get a kind of $p$-adic Jacquet-Langlands functoriality between definite unitary groups and indefinite unitary groups. In some cases we can associate PEL type Shimura varieties to the later group. Then we have more geometric tools to study its overconvergent automorphic forms. For example, Fabrizio Andreatta, Adrian Iovita and Vincent Pilloni developed such $p$-adic theory for Siegel Shimura varieties in \cite{andreatta2015p}. Later Xu Shen generalized their construction to certain compact unitary Shimura varieties in \cite{shen2016p}. On the other hand,   Christopher Birkbeck showed such a $p$-adic Jacquet-Langlands functoriality between eigenvarieties of Hilbert modular forms and definite quaternion algebras over a totally real field in \cite{birkbeck2019jacquet}. I hope to get similar generalizations between unitary groups. After that we may apply results in this paper to study overconvergent automorphic forms on unitary Shimura varieties. For example, Newton applied his $p$-adic level raising results to study certain local global compatibility problems in \cite{newton2015towards}.

Finally,  we discuss further generalization to other groups. As we mentioned after the abelian Ihara lemma, that theorem may be extended to reductive group $G$ over $\Q$ with $G(\R)$ being compact and $G^{der}$ being simply connected. In particular, this  includes all definite unitary groups over any totally real fields. If further there exists a prime $l \neq p$ such that $G(\Q_l)$ has reduced rank one,  then the proof of level raising results should also works. In particular, this applies to all definite $U(3)$ over any totally real fields (like \cite{bellaiche2006u3}). The Bruhat-Tits tree for such $G(\Q_l)$ is either homogeneous or bi-homogeneous. And their unramified principal series is completed classified by \cite{choucroun1994analyse} (see chapter 2). Therefore they are not complicated than $U(3)(\Q_l)$. And our argument works in that general setting. Moreover, the rank one reductive group over a non-Archimedean field is classified by \cite{carbone2001classification}.

However, the higher rank (at $l$) cases is  much more difficult. In fact, even if  in  the classical (modulo $p$) setting, such generalization of Ihara lemma is still open.  Clozel, Harris and Taylor proposed a conjecture about generalizations of Ihara lemma to definite unitary groups $U(n)$ over \textbf{split} primes (thus locally isomorphic to $GL_n$) in \cite{clozel2008automorphy}. This is still quite challenging when $n>2$.  And the $p$-adic level raising problem seems harder. To study such intersection points, we first need to define suitable components. When the local rank is higher, we have to deal with more kinds of components (instead of just \textit{old} or \textit{new}). The local situation in next section already shows such thing. Thus the first task is to define each kind of locus suitably and we have to guarantee that each resulting locus is a union of irreducible  components (which is not obvious). After such constructions, we need to find generalizations of  $p$-adic Ihara lemma, which is more difficult.

\subsection{local analogues}
\label{localpicture}
Finally we discuss a local analogue of such global intersection behaviour. The intuition is to relate degenerate principal series  to certain intersection points on some moduli spaces. In local situation, we will consider the moduli space of tame $L$-parameters. Then the local picture is  easier than the global setting. Along this discussion, we will also see that our definition of old forms and  new forms is more natural than \cite{bellaiche2006u3}.

We first follow section 2 of \cite{hellmann2023derived} to introduce this moduli space. Still let $l \neq p$ be two  primes. Consider the  reductive group over $\Q_l$ (indeed $\Q_l$ can be replaced by any $l$-adic field), $G(\Q_l)=GL(n,\Q_l)$. Its dual group is $GL_n$ and we can identify the $L$-dual group just with this dual group $GL_n$ (as the Galois action is trivial). Let $gl_n$ denote the Lie algebra for $GL_n$ and let $C$ denote an algebraic closed field with characteristic zero (like $\mathbb{C}$ or $\overline{\Q_p}$). Over the field $C$, let $X_{\widehat{G}}$ denote the moduli scheme representing the functor \[R\longrightarrow \{(\phi,N)\in (GL_n\times gl_n)(R)|Ad(\phi)(N)=lN.\}\] on the category of $C$-algebras.

See \cite{hellmann2023derived} (about split reductive groups) for more details and \cite{dat2020moduli} (about more general reductive groups)for a vast generalization of such moduli spaces.

By proposition 2.1 of \cite{hellmann2023derived}, the irreducible components of $X_{\widehat{G}}$ are in bijection with the set of $GL_n$-orbits on the nilpotent cone of $gl_n$. For such an orbit $[N]$, let $X_{\widehat{G}}^{[N]}$ denote the corresponding irreducible components. If $N=0$, we also denote the (unramified) irreducible component as $X_{\widehat{G}}^{un}$.

For any Satake parameter $s \in GL_n(C)$ (more precisely, we pick up a representative inside the $GL_n(C)$-conjugacy class), let $I(s)$ denote the corresponding principal series (via normalized parabolic induction). Recall that $s$ is degenerate if and only if $I(s)$ is reducible. And we can always associate such Satake parameter to a point inside $X_{\widehat{G}}^{un}$ by \[s\mapsto (s,0).\]

We have the following observation:

\begin{prop}
If the Satake parameter $s$ is degenerate, then the corresponding point $(s,0)$ inside $X_{\widehat{G}}^{un}$ will lie at some other irreducible components.

\end{prop}

\begin{proof}
If $s$ is  degenerate, then $I(s)$ is reducible. Suppose its Jordan-Holder factors are $\{\pi_{j}|0\leq j \leq m\}$. And we use  $\pi_0$ to denote the unique unramified representation. Under the local Langlands correspondence, $\pi_0$ will correspond to $GL_n$-orbit of $(s,0)$. Other $\pi_j$ will correspond to certain $GL_n$ orbit (pick up a representative $(s_j,N_j)$) similarly.

By the parabolic induction functoriality, each $s_j$ is conjugated to $s$. This functoriality is well known for $GL(n)$, for example see theorem 1.2(b) of \cite{scholze2013local}. For general reductive groups, this functoriality is expected but not fully established, see conjecture 5.2.2 of \cite{haines2014stable}. Then we can assume $s_j=s$ for each $j$.

For any $j\geq 1$, as $\pi_j$ is ramified, the monodromy $N_j$ is nonzero. The irreducible component $X_{\widehat{G}}^{[j]}=X_{\widehat{G}}^{[N_j]}$ is different from $X_{\widehat{G}}^{un}$. For each positive $j$, there exists a one parameter subgroup \[\psi_j: GL_1 \longrightarrow GL_n,\] such that \[Ad(\psi_j(t))(N_j)=tN_j.\] Then consider the following map $\rho_j: GL_1 \longrightarrow X_{\widehat{G}}$ sending $t$ to $(s,tN_j)$. As each $tN_j$ lies in the same $GL_n$ orbit of $N_j$, the point $\rho_j(t)$ still lies in the irreducible component $X_{\widehat{G}}^{[j]}$. Then the closure of its image $\overline{im(i)}$ also belongs to $X_{\widehat{G}}^{[j]}$. From the construction, obviously the point $(s,0)$ lies in such closure. So $(s,0)$ also lies in the irreducible component $X_{\widehat{G}}^{[j]}$ for each positive $j$.

\end{proof}

Therefore, in the case of $GL(n)$, the local picture is very nice.  However, this conclusion won't always hold. For other reductive groups, there exists degenerate Satake parameters which even \textbf{don't} contribute to intersection points. The main reason is that $GL(n)$ has only one conjugacy class of maximal open subgroups while other groups may have more such conjugacy classes. For that $G$, its unramified $L$-packet may have more than one element.

For instance, such phenomena  already happens for $SL(2,\Q_l)$. Its dual group is $PGL_2$ (still with trivial Galois action), and we can define the moduli space $X_{\widehat{G}}$ similarly. The group $SL(2,\Q_l)$ has two conjugacy class of maximal open compact subgroups. One is $K_0=SL(2,\Z_l)$ and the other one is $K_1=diag(l,1)K_0diag(l^{-1},1)$. And there is a degenerate Satake parameter $s=diag(1,1)$, $I(s)$ has two Jordan-Holder factor $\pi_0$ and $\pi_1$, while $\pi_0$ is unramified respect to $K_0$ and $\pi_1$ is unramified respect to $K_1$. Then under the local Langlands correspondence, they should correspond to the same  orbit represented by $(s,0)$ inside $X_{\widehat{G}}$. Thus this unramified $L$-packet has two elements and through direct computation, this point $(s,0)$ is not an intersection point.

In some sense, the group $U(3,\Q_l)$ is a similar example. Its dual group is $GL_3$, but with a non-trivial Galois action. Its Langlands dual group is $GL_3 \rtimes \{\pm1\}$. The construction of \cite{hellmann2023derived} is for split reductive groups and we should use the general construction in \cite{dat2020moduli}. Here we only mention some results about principal series of $U(3,\Q_l)$. Section 3.6 of \cite{bellaiche2006u3} lists a classification. They use a complex number $\alpha$ to denote the Satake parameter. Under their notations, there are three kinds of unramified principal series:

$ \bullet$ If $\alpha \neq l^{\pm2}$ and $\alpha \neq -l^{\pm1}$, then the principal series $I(s)$ is irreducible;

$\bullet$ If $\alpha = l^{\pm 2}$ ($\Leftrightarrow$the eigenvalue of $T_l$ is $l(l^3+1)$), then $I(s)$ has two Jordan-Holder factors, one is a character and another one is the Steinberg representation;

$\bullet$ If $\alpha= - l^{\pm1}$ ($\Leftrightarrow$the eigenvalue of $T_l$ is $-(l^3+1)$), then $I(s)$ has two Jordan-Holder factors, each has a nonzero invariant vector under certain maximal open compact subgroups.

The second situation gives another motivation to study the abelian Ihara lemma: one dimensional unramified representation ('\textit{old}') is closely related with the Steinberg representation ('\textit{new}') at such case. And we expect such point to be intersection points. But the degenerate Satake parameter in the third situation won't contribute to intersection points. In \cite{bellaiche2006u3}, they use $L_0 \oplus L_0 \longrightarrow M$ to define level raising maps, and the level changing matrix (proposition 3.5.4 of \cite{bellaiche2006u3}) has determinant $(l^4+l-T_l)(l^3+1+T_l)$. Then both $l(l^3+1)$ and $-(l^3+1)$ are 'singular' eigenvalue. However, only eigenvalue of the form $l(l^3+1)$ corresponds to intersection points in geometry, justifying our choice  of $L_0 \oplus L_1$ instead of $L_0 \oplus L_0$.

Finally we remark that even if for the simplest example $GL(2,\Q_l)$, not all intersection points come from degenerate Satake parameters. The reason is that under the local Langlands correspondence, the resulting Weil-Deligne representation is Frobenius semisimple. The Satake  parameters only runs over semisimple elements while intersection points may not satisfy such property. The geometry of $X_{\widehat{G}}$ is more delicate than the study of degenerate principal series. And it is more natural to study the quotient stack $[X_{\widehat{G}}/\widehat{G}]$ (in particular for categorical local Langlands), which is more complicated.

\newpage

\bibliographystyle{plain}
\bibliography{reference}

\end{document}